\documentclass[12pt]{amsart}

\usepackage{amsfonts,amssymb,amsmath,textcomp,amsthm}
\usepackage{color}
\usepackage{hyperref}
\usepackage{cancel}
\usepackage{amstext} 
\usepackage{array}   
\usepackage{tikz-cd}
\usepackage{float}
\newcolumntype{C}{>{$}c<{$}} 

\usepackage{tabularx}
\usepackage{amsmath}
\usepackage{longtable,booktabs}

\usepackage{caption}
\newtheorem{theorem}{Theorem}[section]
\theoremstyle{plain}

\newtheorem{corollary}[theorem]{Corollary}

\newtheorem{lemma}[theorem]{Lemma}

\newtheorem{proposition}[theorem]{Proposition}

\numberwithin{equation}{section}

\theoremstyle{definition}
\newtheorem{definition}[theorem]{Definition}

\newtheorem{remark}[theorem]{Remark}

\makeatletter
\let\@setaddresses\relax
\makeatother

\begin{document}
	
	
	\title[Groups with triality and code loops]{Construction of Groups with Triality \\ and their Corresponding Code Loops}
	
	
	
	
	\author[R.~Miguel Pires]{Rosemary Miguel Pires}
	\thanks{Departamento de Matemática, Instituto de Ciências Exatas,
		Universidade Federal Fluminense, Volta Redonda, Rio de Janeiro, 27213-145, Brazil.
		\textit{Email}: \texttt{rosemarypires@id.uff.br}.}
	
	\author[A.~Grishkov]{Alexandre Grishkov}
	\thanks{Departamento de Matemática, Universidade de São Paulo, 
		Caixa Postal 66281, São Paulo, 05311-970, Brazil, 
		and Omsk State University n.a.\ F.\ M.\ Dostoevskii, Omsk, Russia.
		\textit{Email}: \texttt{grishkov@ime.usp.br}.}

	\author[R.L.~Rodrigues]{Rodrigo Lucas Rodrigues}
	\thanks{Departamento de Matemática, Universidade Federal do Ceará, Campus do Pici, Bloco 914, 60455-760, Fortaleza, Brazil.
		\textit{Email}: \texttt{rodrigo@mat.ufc.br}.}
	
	\author[M.~Rasskazova]{Marina Rasskazova}
	\thanks{Centro de Matem\'atica, Computa\c c\~ao  e Cogni\c c\~ao, Universidade Federal do ABC, Santo André, S\~ao Paulo, 09280-560, Brazil, and Siberian State Automobile and Highway University,  Omsk, Omsk Oblast, 644080, Rússia
		\textit{Email}: \texttt{marinarasskazova@yandex.ru}.}

	\begin{abstract}

		We generalize the global construction of code loops introduced by Nagy~\cite{Nagy}, which is based on the connection between Moufang loops and groups with triality.
		This follows from the construction of a nilpotent group \( G_n \) of class~3 with triality and \(2n\) generators, based on embeddings of \( G_n \) into direct products of copies of \( G_3 \).
		In the finite case, where \( G_n \) is a group such that \( |G_n| = 2^{4n+m} \) with \( n \geq 3 \) and \( m = 3 {n \choose 2} + 2 {n \choose 3} \), we prove that the corresponding Moufang loop is the free loop \( \mathcal{F}_n \) with \( n \) generators in the variety \( \mathcal{E} \) generated by code loops. 
		The result depends on a construction similar to that of \( G_n \), namely, embedding \( \mathcal{F}_n \) into direct products of copies of \( \mathcal{F}_3 \), the free code loop associated with \( G_3 \).
		
	\end{abstract}
	
	
	\maketitle

	\section{Introduction}  
	A set $L$ with a binary operation $(x,y) \mapsto x\cdot y = xy$ is called a \textit{loop} if the equation $xy = z$ can be solved uniquely if two of the variables are given and there exists an identity element $1$ with $1x = x = x1$ for every $x \in L$. We use the notation $x^{-1}$ to denote the unique 2-sided inverse of $x$. 
	
	A loop $L$ is said to be Moufang if it satisfies any of the three (equivalent) Moufang identities:
	\begin{table}[h]
		\centering
		\begin{tabular}{ccc}
			$((xy)x)z = x(y(xz)),$    &  $((xy)z)y = x(y(zy)),$ & $(xy)(zx) = (x(yz))x,$
		\end{tabular}
	\end{table}
	\\
	the left Moufang identity, the right Moufang identity, and the middle Moufang identity, respectively. For basic definitions and properties of Moufang loops, we refer the reader to \cite{Bruck}.

	In a loop $L$, the right and left translation maps $R_x$ and $L_x$, $x \in L$, are defined by 
	$a R_x = ax$ and $a L_x = xa$ for all $a \in L$. The set of left and right translation maps generates a group ${\rm Mlt}(L)$, called the multiplication group of $L$, which is a subgroup of the symmetric group on the set $L$. We denote it by
	${\rm Mlt}(L) = \langle L_x, R_x \ | \ x \in L \rangle.$
	
	Let $a, b$ and $c$ three elements of a loop $L$. The (loop) commutator of $a$ and $b$ is the unique element $(a, b)$ of $L$ which satisfies $ab = (ba) (a, b)$ and the (loop) associator of $a, b$ and $c$ is the unique element $(a, b, c)$ of $L$ which satisfies $(ab)c = {a(bc)}(a,b,c)$. The nucleus of $L$ is
	$$N(L) = \{a \in L \ | \ (a, x, y) = (x, y, a) = (x, a, y) = 1, \ {\rm for} \ {\rm all} \ x, y \in L \}.$$
	and the center of $L$ is
	$$Z(L) = \{x \in N(L) \ | \ (a, x) = 1, \ {\rm for} \ {\rm all} \ a \in L  \}.$$

	Glauberman \cite{Glauberman} pointed out that for Moufang loops $L$ with a trivial nucleus, the multiplication group ${\rm Mlt}(L)$ admits a
	certain dihedral group of automorphisms of order 6. 
	In 1978, Doro \cite{Doro} introduced the concept of a group with triality. 
	We call $G$ a group with triality $S$ if $G$ is a group, 
	$S$ is a subgroup of the automorphism group of $G$ such that $S = \langle \sigma, \rho \ | \ \sigma^2 = \rho^3 = (\sigma \rho)^2 = 1 \rangle \cong S_3$ and for all $g \in G$ the triality identity $[g, \sigma] [g, \sigma]^{\rho} [g, \sigma]^{\rho^2} = 1$ holds, where $[g, \sigma] = g^{-1} g^{\sigma}$. He also proved that from any group $G$ with triality $S$, it is possible to construct a Moufang loop $L$ in which the group $G$ has a permutation representation isomorphic to ${\rm Mlt}(L)$; every Moufang loop arises in this way from a group with triality; associative Moufang loops correspond to wreath products or to closely related constructions, while simple non-associative Moufang loops correspond to simple groups with triality. Since the discovery of
	the relation between groups with triality and Moufang loops, there appeared new ways of constructing Moufang loops using groups
	with triality.
	
	In \cite{Paige}, Paige defines, for each finite field $GF(q)$, a finite simple Moufang loop $M(q)$ which is nonassociative. In 1987, Liebeck \cite{Liebeck} proved that if $M$ is a finite simple Moufang loop, then either $M$ is associative (and hence is a finite simple group), or $M$ is isomorphic to one of the loops $M(q)$. He emphasized that the existence of the simple Moufang loops $M(q)$ is of course intimately
	related to the fact that the simple groups ${\rm Mlt}(M(q))$ are groups with triality, and he used the classification of finite simple
	groups to prove that the only finite simple groups with triality are the simple groups ${\rm Mlt}(M(q))$.
	
	Hall and Nagy \cite{HallNagy} extended the partial relationship between Moufang loops and groups with triality, given by Doro, by showing that the following concepts are equivalent: groups with triality and trivial center, Moufang 3-nets, latin square designs in which every point is the center of an automorphism, and isotopy classes of Moufang loops. Using this new approach, they also gave a simple proof to a theorem due to Doro.
	
	Grishkov and Zavarnitsine (\cite{GrishkovZavarnitsine05}, \cite{GrishkovZavarnitsine06}) determined the correspondence between the maximal subloops of a given Moufang loop and some subgroups of the corresponding group with triality, and 
	described all possible, in a sense, groups with triality associated with a given Moufang loop, introducing several universal groups with triality and discussing their properties.  They also presented another way to construct a Moufang loop corresponding to a given group with triality.
	
	Nagy \cite{Nagy} used that Moufang loops can be
	equivalently given by the specific group theoretical concept of groups with triality to present a global construction for the loop. 
	Gagola III \cite{GagolaIII10} obtained information about the smallest group with triality associated with a finite Moufang loop that is solvable.
	
	In 2012, 2013 and 2016, groups with triality were used by Gagola III \cite{GagolaIII12} to show that any Moufang loop obtained by a cyclic extension of an abelian group of odd order, by a cyclic group with an order coprime to 3, is necessarily a group; and by 
	Grishkov and Zavarnitsine \cite{GrishkovZavarnitsine13} 
	to construct a series of nonassociative Moufang loops, where certain members of these series contain an abelian normal subloop with the corresponding quotient being a cyclic group. In particular, they gave a new series of examples of finite abelian-by-cyclic Moufang loops, the smallest having order $3.2^6$ and some of the examples were shown to be embeddable into a Cayley algebra. Furthermore, 
	they obtained some general multiplication formulas in Moufang loops, constructed Moufang extensions of abelian groups, and described the structure of minimal extensions for finite simple Moufang loops over abelian groups.

	A Moufang loop $L$ such that $L/A$ is an elementary abelian $2$-group for some central subgroup $A$ of order $2$ is called a code loop. The notion and construction of code loops were introduced by Griess \cite{Griess}. Chein and Goodaire \cite{CG} proved that finite code loops may be characterized as Moufang loops $L$ such that $|L^2| \leq 2$. They also proved that these code loops have a unique nonidentity commutator, a unique nonidentity associator, and a unique nonidentity square. In \cite{GP}, Grishkov and R. M. Pires presented a classification of code loops with $3$ and $4$ generators and the corresponding groups of automorphisms. For this purpose they introduced the notion of a \emph{characteristic vector}. More precisely, if  $\{a_1,a_2,a_3\ldots,a_n\}$ is a mininal set of generators for a code loop $L$, then the characteristic vector is defined by $\overline{\lambda}=(\lambda_1,\ldots,\lambda_n,\lambda_{1,2}, \ldots,\lambda_{n-1,n},\lambda_{1,2,3},\ldots,\lambda_{n-2,n-1,n})$ where $\lambda_i$, $\lambda_{i,j}$, $\lambda_{i,j,k}\in\{0,1\}$, with $i<j<~k$, and the values of these parameters are determined by the relations $a_i^2=(-1)^{\lambda_i}$, $[a_i,a_j]=(-1)^{\lambda_{i,j}}$ and $(a_i,a_j,a_k)=(-1)^{\lambda_{i,j,k}}$.
	
	Furthermore, in \cite{GP}, the authors constructed a free Moufang loop in the variety generated by code loops and proved that code loops with $n$ generators can be characterized as a homomorphic image of a free Moufang loop with $n$ generators.

	The above-mentioned articles established the connection between groups with triality and Moufang loops, as well as several important constructions and classifications. 
	Our purpose in this paper is to present a construction of a nilpotent group \( G_n \) of class~3 with triality and \( 2n \) generators, generalizing the case of \( G_3 \). 
	The group \( G_n \) is defined as an infinite nilpotent group of class~3. In the finite case, obtained by adding further relations, we determine its order and show that the corresponding Moufang loop is precisely the free Moufang loop \( \mathcal{F}_n \) in the variety \( \mathcal{E} \) generated by code loops. This work generalizes the construction introduced by Nagy~\cite{Nagy}: while Nagy’s approach produces, for each given code, a specific group with triality whose associated loop is the corresponding code loop, our construction provides a universal group \( G_n \) whose associated loop \( \mathcal{F}_n \) is the free object in the same variety, from which all code loops of rank~\( n \) arise as quotients.
	
	The paper is organized as follows. 
	In Section~\ref{section:group.with.triality}, we construct a nilpotent group \( G_n \) of class~3 with triality and \( 2n \) generators, and show how \( G_n \) can be embedded into a direct product of copies of \( G_3 \). 
	In Section~\ref{section:variety.code.loops}, we review some concepts and results related to the variety \( \mathcal{E} \) generated by code loops. 
	In Section~\ref{section:loop.corresponding}, we present the Moufang loop corresponding to \( G_n \) in the finite case, and in Section~\ref{section:comparison}, we establish its relation to the group introduced by Nagy~\cite{Nagy}, showing that our approach generalizes his result.
	

	
	
	

	\section{A group with triality}\label{section:group.with.triality}
	
	In this section, we provide a construction of a nilpotent group of class~3 with triality. 
	For a group \( G \) and elements \( a, b \in G \), the commutator is defined by 
	\([a,b] = a^{-1} b^{-1} a b\). 
	The derived subgroup (or commutator subgroup) of \( G \) is denoted by 
	\([G,G] = \langle [a,b] : a,b \in G \rangle.\)
	
	Recall that a group \( G \) is said to be \emph{nilpotent of class~\(n\)} 
	if its lower central series
	\[
	\gamma_1(G) = G, \qquad 
	\gamma_{i+1}(G) = [\gamma_i(G), G] \quad (i \ge 1)
	\]
	satisfies \( \gamma_{n+1}(G) = 1 \) but \( \gamma_n(G) \ne 1 \).

	Let $G_n$ be a nilpotent group of class $3$ with $2n$ generators $\left\{a_1,\ldots,a_n, b_1, \ldots,b_n \right\}$.  For $1 \leq i,j \leq n$, we denote 
	$u_{ij} := [a_i,a_j]$, $v_{ij} := [b_i,b_j]$ and $p_{ij} := [a_i,b_j]$.
	The defining relations of $G_n$, valid for all $1 \leq i,j,k \leq n$, are the following:

	\begin{itemize}
		\item[$(1)$] \refstepcounter{enumi} \label{rel.1} 
		$[a_i, b_i] = 1$, $[a_i^2, G_n] = [b_i^2, G_n] = 1$. 
		
		\item[$(2)$] \refstepcounter{enumi} \label{rel.2} 
		$u_{ij} = u_{ji}$, $v_{ij} = v_{ji}$, $u_{ij}^2 = v_{ij}^2 = 1$, for $i < j$.
		
		\item[$(3)$] \refstepcounter{enumi} \label{rel.3} 
		$[u_{ij}, G_n] = [v_{ij}, G_n] = 1$, for $i < j$.
		
		\item[$(4)$] \refstepcounter{enumi} \label{rel.4} 
		$p_{ij} = p_{ji}$, $p_{ij}^2 = 1$, for $i < j$.
		
		\item[$(5)$] \refstepcounter{enumi} \label{rel.5} 
		$[p_{ij}, a_k] := z_{ijk}$, and $z_{ijk} = z_{jik} = z_{kij}$,  for $|\{i,j,k\}|=3$.
		
		\item[$(6)$] \refstepcounter{enumi} \label{rel.6} 
		$[p_{ij}, b_k] := t_{ijk}$, and $t_{ijk} = t_{jik} = t_{kij}$, for $|\{i,j,k\}|=3$.
		
		\item[$(7)$] \refstepcounter{enumi} \label{rel.7} 
		$t_{ijk}^2 = z_{ijk}^2 = 1$, for $|\{i,j,k\}|=3$. 
		
		\item[$(8)$] \refstepcounter{enumi} \label{rel.8} 
		$[z_{ijk},G_n]=[t_{ijk},G_n]=1$, for $|\{i,j,k\}|=3$. 
		
	\end{itemize}
	
	
	
	
	Observe that, in fact, by construction and using relations above, $G_n$ is nilpotent of class $3$.
	

	In what follows, we focus on the case $n=3$ and recall a normal form for the elements of $G_3$.

	\begin{remark}\label{rem.normalform}
		Let $G = G_3$ be the nilpotent group generated by $\left\{a_1,a_2, a_3, b_1, b_2, b_3 \right\}$ satisfying the relations \ref{rel.1} through \ref{rel.8} above. Note that any element $x \in G$ can be written as 
		\begin{equation} 
			x = a_{1}^{\alpha_1}a_{2}^{\alpha_2}a_{3}^{\alpha_3} b_{1}^{\alpha_4}b_{2}^{\alpha_5}b_{3}^{\alpha_6} u_{12}^{\alpha_{12}^{1}}u_{13}^{\alpha_{13}^{1}}u_{23}^{\alpha_{23}^{1}
			}v_{12}^{\alpha_{12}^{2}}v_{13}^{\alpha_{13}^{2}}v_{23}^{\alpha_{23}^{2}}p_{12}^{\alpha_{12}^{3}}p_{13}^{\alpha_{13}^{3}}p_{23}^{\alpha_{23}^{3}} z_{123}^{\alpha_{123}^{1}} t_{123}^{\alpha_{123}^{2}},
		\end{equation}
		where $\alpha_l \in \mathbb{Z}$ for $l=1,\dots,6$, and the other exponents are equal to $0$ or $1$.\\
		
		Moreover, the element $\mu = u_{12}^{\alpha_{12}^{1}}u_{13}^{\alpha_{13}^{1}}u_{23}^{\alpha_{23}^{1}
		}v_{12}^{\alpha_{12}^{2}}v_{13}^{\alpha_{13}^{2}}v_{23}^{\alpha_{23}^{2}}z_{123}^{\alpha_{123}^{1}} t_{123}^{\alpha_{123}^{2}}$ belongs to the center $Z(G)$. Therefore, the element $x$ can be rewritten as  
		$x = g\mu$, with $$g = a_{1}^{\alpha_1}a_{2}^{\alpha_2}a_{3}^{\alpha_3} b_{1}^{\alpha_4}b_{2}^{\alpha_5}b_{3}^{\alpha_6}p_{12}^{\alpha_{12}^{3}}p_{13}^{\alpha_{13}^{3}}p_{23}^{\alpha_{23}^{3}}.$$

	\end{remark}

	Consider $x = g\mu \in G$, with $g$ and $\mu$ defined as in Remark~\ref{rem.normalform}. 
	For convenience, we introduce the following matrix notation to represent $g$:
	
	$$g :=\begin{pmatrix}
		\alpha_1 & \alpha_2 & \alpha_3 \\ 
		\alpha_4 & \alpha_5 & \alpha_6 \\
		\alpha_{12}^3 & \alpha_{13}^3 & \alpha_{23}^3
	\end{pmatrix}.$$

	Using the matrix notation above, the product of two elements in $G$ can be described as follows.
	\begin{lemma} \label{lemma.product}
		Let $G$ be a group as stated in Remark $\ref{rem.normalform}$, and let $g_1, g_2 \in G$ be given by 
		\[
		g_1 = 
		\begin{pmatrix}
			\alpha_1 & \alpha_2 & \alpha_3 \\ 
			\alpha_4 & \alpha_5 & \alpha_6 \\
			\alpha_{12}^3 & \alpha_{13}^3 & \alpha_{23}^3
		\end{pmatrix}, 
		\qquad
		g_2 = 
		\begin{pmatrix}
			\beta_1 & \beta_2 & \beta_3 \\ 
			\beta_4 & \beta_5 & \beta_6 \\
			\beta_{12}^3 & \beta_{13}^3 & \beta_{23}^3
		\end{pmatrix},
		\]
		where $\alpha_i, \beta_i \in \mathbb{Z}$ for $i = 1,\dots,6$, and the other exponents are equal to $0$ or $1$. 
		Then
		\[
		g_{1} g_{2} = 
		\begin{pmatrix}
			\alpha_1 + \beta_1 & \alpha_2 + \beta_2 & \alpha_3 + \beta_3 \\ 
			\alpha_4 + \beta_4 & \alpha_5 + \beta_5 & \alpha_6 + \beta_6 \\
			\alpha_{12}^{3} + \beta_{12}^{3} + \alpha_4 \beta_2 + \alpha_5 \beta_1 & 
			\alpha_{13}^{3} + \beta_{13}^{3} + \alpha_4 \beta_3 + \alpha_6 \beta_1 & 
			\alpha_{23}^{3} + \beta_{23}^{3} + \alpha_5 \beta_3 + \alpha_6 \beta_2
		\end{pmatrix} \mu,
		\]
		where addition in the third row is taken modulo $2$, and  $\mu \in Z(G)$ is given by
		\[
		\mu \;=\; 
		u_{12}^{\alpha_2 \beta_1}\,u_{13}^{\alpha_3 \beta_1}\,u_{23}^{\alpha_3 \beta_2}\,
		v_{12}^{\alpha_5 \beta_4}\,v_{13}^{\alpha_6 \beta_4}\,v_{23}^{\alpha_6 \beta_5}\,
		z_{123}^{\varepsilon_z}\, t_{123}^{\varepsilon_t},
		\]
		where
		\[
		\begin{aligned}
			\varepsilon_z &:= \alpha_{23}^{3}\beta_1 + \alpha_{13}^{3}\beta_2 + \alpha_{12}^{3}\beta_3
			+ \alpha_5 \beta_1 \beta_3 + \alpha_6 \beta_1 \beta_2 + \alpha_4 \beta_2 \beta_3,\\
			\varepsilon_t &:= \alpha_{23}^{3}\beta_4 + \alpha_{13}^{3}\beta_5 + \alpha_{12}^{3}\beta_6
			+ \alpha_4 \alpha_5 \beta_3 + \alpha_5 \alpha_6 \beta_1 + \alpha_4 \alpha_6 \beta_2 \\
			&\quad + \alpha_6 \beta_2 \beta_4 + \alpha_5 \beta_3 \beta_4 + \alpha_6 \beta_1 \beta_5
			+ \alpha_4 \beta_3 \beta_5 + \alpha_5 \beta_1 \beta_6 + \alpha_4 \beta_2 \beta_6.
		\end{aligned}
		\]
	\end{lemma}
	
	\begin{proof}
		The result follows by applying the defining commutator and the square relations of $G$. 
		Since the calculation is straightforward, but long, we omit the details. 
	\end{proof}

	We now introduce the group $\mathbb{P}_n$, 
	which gives a natural coordinate description of $G_n$. 
	This will allow us to prove that $G_n$ and $\mathbb{P}_n$ are indeed isomorphic.
	
	\subsection{Group isomorphic to $G_n$}

	For $n\geq 3$, set $m = 3 {n \choose 2} + 2 {n \choose 3}$, and consider $\mathbb{P}_n = \mathbb{Z}^{2n} \oplus \mathbb{F}_2^m$. Let:
	\vspace{-0.2in}
	
	
	
	\begin{equation}\label{def:alpha}
		\begin{aligned}
			\overline{\alpha} = (&\alpha_1, \ldots, \alpha_n;\; 
			\alpha_{n+1}, \ldots, \alpha_{2n};\;
			\underbrace{\alpha_{12}^1, \alpha_{13}^1, \ldots, \alpha_{1n}^1, 
				\alpha_{23}^1, \ldots, \alpha_{n-1,\, n}^1}_{{n \choose 2}}; \\
			&\underbrace{\alpha_{12}^2, \alpha_{13}^2, \ldots, \alpha_{1n}^2, 
				\alpha_{23}^2, \ldots, \alpha_{n-1,\, n}^2}_{{n \choose 2}}; \underbrace{\alpha_{12}^3, \alpha_{13}^3, \ldots, \alpha_{1n}^3, 
				\alpha_{23}^3, \ldots, \alpha_{n-1,\, n}^3}_{{n \choose 2}}; \\
			&\underbrace{\alpha_{123}^1, \alpha_{124}^1, \ldots, 
				\alpha_{n-2,\, n-1,\, n}^1}_{{n \choose 3}};\;
			\underbrace{\alpha_{123}^2, \alpha_{124}^2, \ldots, 
				\alpha_{n-2,\, n-1,\, n}^2}_{{n \choose 3}})
		\end{aligned}
	\end{equation}
	and let $\overline{\beta}$ be defined similarly.
	We define a product on $\mathbb{P}_n$ by:
	\vspace{-0.1in}
	
	$$\overline{\alpha} \cdot \overline{\beta} = \overline{\gamma} \in \mathbb{P}_{n},$$
	where
	\vspace{-0.3in}

	
	
	\begin{equation}\label{def:gamma}
		\begin{aligned}
			\overline{\gamma} &= \overline{\alpha} + \overline{\beta} + \overline{\tau}, \\[6pt]
			\overline{\tau} = (&\underbrace{0,\ldots,0}_{n};\;
			\underbrace{0,\ldots,0}_{n};\;
			\tau_{12}^1, \ldots, \tau_{n-1,\, n}^1;\;
			\tau_{12}^2, \ldots, \tau_{n-1,\, n}^2;\;
			\tau_{12}^3, \ldots, \tau_{n-1,\, n}^3; \\
			&\tau_{123}^1, \ldots, \tau_{n-2,\, n-1,\, n}^1;\;
			\tau_{123}^2, \ldots, \tau_{n-2,\, n-1,\, n}^2).
		\end{aligned}
	\end{equation}

	The coordinates of $\overline{\tau}$, starting from the $(2n+1)$-th position, are given modulo $2$ by:
	
	\begin{equation*}\label{def:tau}
		\begin{aligned}
			&\text{For } i<j: \\
			&\quad \tau_{ij}^1 = \alpha_j \beta_i, \\
			&\quad \tau_{ij}^2 = \alpha_{j+n} \beta_{i+n}, \\
			&\quad \tau_{ij}^3 = \alpha_{i+n} \beta_{j} + \alpha_{j+n}\beta_{i}; \\[6pt]
			&\text{For } i<j<k: \\
			&\quad \tau_{ijk}^1 = \alpha_{ij}^3 \beta_{k} + \alpha_{ik}^3 \beta_{j} + \alpha_{jk}^3 \beta_{i} 
			+ \alpha_{i+n} \beta_{j}\beta_{k} + \alpha_{j+n} \beta_{i}\beta_{k} + \alpha_{k+n} \beta_{i}\beta_{j}, \\
			&\quad \tau_{ijk}^2 = \alpha_{ij}^3 \beta_{k+n} + \alpha_{ik}^3 \beta_{j+n} + \alpha_{jk}^3 \beta_{i+n} 
			+ \alpha_{i+n}\alpha_{j+n}\beta_{k} + \alpha_{i+n}\alpha_{k+n}\beta_{j} + \alpha_{j+n}\alpha_{k+n}\beta_{i} \\
			&\quad\qquad + \alpha_{i+n}(\beta_j \beta_{k+n} + \beta_{j+n}\beta_{k})
			+ \alpha_{j+n}(\beta_i \beta_{k+n} + \beta_{i+n}\beta_{k})
			+ \alpha_{k+n}(\beta_i \beta_{j+n} + \beta_{i+n}\beta_j).
		\end{aligned}
	\end{equation*}
	

	The following remark presents a simpler way to compute the product of two elements in $\mathbb{P}_3$.
	
	\begin{remark}\label{remark:product_P3}
		Let $\overline{\alpha}, \overline{\beta} \in \mathbb{P}_3$ and put $\overline{\gamma}=\overline{\alpha}\overline{\beta}$.
		In Table~\ref{tab:mult.p3}, we present in columns~2, 3 and~4 the coordinates of $\overline{\alpha}$, $\overline{\beta}$, and $\overline{\gamma}$, respectively.
		In the 5th column, we make explicit the corresponding terms $\tau_{ij}^{k}$ and $\tau_{ijk}^{\ell}$ as given in~\eqref{def:gamma}.
		This table gives a simple rule to compute the product of two arbitrary elements of $\mathbb{P}_3$. 
		
		For convenience, and to facilitate the visualization of the coordinates of an arbitrary element $\overline{x}\in\mathbb{P}_3$, 
		we write $\overline{x}=(x_1,\dots,x_{17})$ according to the order in the first column 
		(in particular, $x_7=x_{12}^{1}$, $x_8=x_{13}^{1}$, and so on, up to $x_{17}=x_{123}^{2}$).
		
		For instance, if $\overline{\alpha}, \overline{\beta}\in \mathbb{P}_3$ and $\overline{\gamma}=\overline{\alpha}\overline{\beta}$ as above, 
		we can write the coordinates of $\overline{\gamma}$ as $(\gamma_1,\dots,\gamma_{17})$, that is,
		$\gamma_1=\alpha_1+\beta_1, \dots, \gamma_6=\alpha_6+\beta_6, \gamma_7=\alpha_{12}^{1}+\beta_{12}^{1}+\tau_{12}^{1}, \gamma_8=\alpha_{13}^{1}+\beta_{13}^{1}+\tau_{13}^{1}$, and so on.
		Entries corresponding to rows $7$--$17$ (the $\mathbb{F}_2$ part) are taken modulo~2.
	\end{remark}

	\begin{table}[H]\centering
		\footnotesize
		\begin{tabular}{l||l|l|l|p{8cm}}
			& $\overline{\alpha}$ & $\overline{\beta}$ & $\overline{\gamma}=\overline{\alpha}\overline{\beta}$ & formulas for $\tau_{ij}^{k}$ and $\tau_{ijk}^{\ell}$ \\ \hline\hline
			1  & $\alpha_1$          & $\beta_1$          & $\alpha_1 + \beta_1$                                  & 0 \\
			2  & $\alpha_2$          & $\beta_2$          & $\alpha_2 + \beta_2$                                  & 0 \\
			3  & $\alpha_3$          & $\beta_3$          & $\alpha_3 + \beta_3$                                  & 0 \\
			4  & $\alpha_4$          & $\beta_4$          & $\alpha_4 + \beta_4$                                  & 0 \\
			5  & $\alpha_5$          & $\beta_5$          & $\alpha_5 + \beta_5$                                  & 0 \\
			6  & $\alpha_6$          & $\beta_6$          & $\alpha_6 + \beta_6$                                  & 0 \\
			7  & $\alpha_{12}^1$     & $\beta_{12}^1$     & $\alpha_{12}^{1}+ \beta_{12}^{1} + \tau_{12}^{1}$     & $\tau_{12}^{1} = \alpha_2 \beta_1$ \\
			8  & $\alpha_{13}^1$     & $\beta_{13}^1$     & $\alpha_{13}^{1}+ \beta_{13}^{1} + \tau_{13}^{1}$     & $\tau_{13}^{1} = \alpha_3 \beta_1$ \\
			9  & $\alpha_{23}^1$     & $\beta_{23}^1$     & $\alpha_{23}^{1}+ \beta_{23}^{1} + \tau_{23}^{1}$     & $\tau_{23}^{1} = \alpha_3 \beta_2$ \\
			10 & $\alpha_{12}^2$     & $\beta_{12}^2$     & $\alpha_{12}^{2}+ \beta_{12}^{2} + \tau_{12}^{2}$     & $\tau_{12}^{2} = \alpha_5 \beta_4$ \\
			11 & $\alpha_{13}^2$     & $\beta_{13}^2$     & $\alpha_{13}^{2}+ \beta_{13}^{2} + \tau_{13}^{2}$     & $\tau_{13}^{2} = \alpha_6 \beta_4$ \\
			12 & $\alpha_{23}^2$     & $\beta_{23}^2$     & $\alpha_{23}^{2}+ \beta_{23}^{2} + \tau_{23}^{2}$     & $\tau_{23}^{2} = \alpha_6 \beta_5$ \\
			13 & $\alpha_{12}^3$     & $\beta_{12}^3$     & $\alpha_{12}^{3}+ \beta_{12}^{3} + \tau_{12}^{3}$     & $\tau_{12}^{3} = \alpha_4 \beta_2 + \alpha_5 \beta_1$ \\
			14 & $\alpha_{13}^3$     & $\beta_{13}^3$     & $\alpha_{13}^{3}+ \beta_{13}^{3} + \tau_{13}^{3}$     & $\tau_{13}^{3} = \alpha_4 \beta_3 + \alpha_6 \beta_1$ \\
			15 & $\alpha_{23}^3$     & $\beta_{23}^3$     & $\alpha_{23}^{3}+ \beta_{23}^{3} + \tau_{23}^{3}$     & $\tau_{23}^{3} = \alpha_5 \beta_3 + \alpha_6 \beta_2$ \\
			16 & $\alpha_{123}^1$    & $\beta_{123}^1$    & $\alpha_{123}^{1}+ \beta_{123}^{1} + \tau_{123}^{1}$  & $\tau_{123}^{1} = \alpha_{12}^{3}\beta_3 + \alpha_{13}^{3}\beta_2 + \alpha_{23}^{3}\beta_1 + \alpha_{4}\beta_2\beta_3 + \alpha_5\beta_1\beta_3 + \alpha_6\beta_1\beta_2$ \\
			17 & $\alpha_{123}^2$    & $\beta_{123}^2$    & $\alpha_{123}^{2}+ \beta_{123}^{2} + \tau_{123}^{2}$  & $\tau_{123}^{2} = \alpha_{12}^{3}\beta_6 + \alpha_{13}^{3}\beta_5 + \alpha_{23}^{3}\beta_4 + \alpha_4\alpha_5\beta_3 + \alpha_4\alpha_6\beta_2 + \alpha_5\alpha_6\beta_1 + \alpha_{4}(\beta_2\beta_6+\beta_5\beta_3) + \alpha_{5}(\beta_1\beta_6+\beta_4\beta_3) + \alpha_{6}(\beta_1\beta_5+\beta_4\beta_2)$
		\end{tabular}
		\caption{Multiplication in $\mathbb{P}_3$}
		\label{tab:mult.p3}
	\end{table}
	\vspace{-0.4cm}
	Using the notation and rules described in the previous remark, we can now prove that $\mathbb{P}_3$ is a group.
	
	\begin{lemma}\label{lemma:p3group}
		$\mathbb{P}_3$ is a group.
	\end{lemma}
	\begin{proof}
		
		Using the notation and rules described in Remark~\ref{remark:product_P3}, 
		we compute the product of three elements to verify associativity.
		Let $\overline{a}, \overline{b}, \overline{c} \in \mathbb{P}_3$, 
		$\overline{x} = \overline{a}\,\overline{b}$, 
		$\overline{y} = \overline{b}\,\overline{c}$, 
		$\overline{r} = (\overline{a}\,\overline{b})\,\overline{c}$, 
		and $\overline{s} = \overline{a}\,(\overline{b}\,\overline{c})$. 
		We need to prove that $\overline{r} = \overline{s}$.
		
		It is clear that $r_i = s_i$ for $i=1,\dots,6$. 
		Table~\ref{tab:assoc.P3} lists the values of $r_i$ for $i=7,\dots,17$. 
		The corresponding values of $s_i$ can be obtained analogously by replacing, 
		in the second column, $x$ with $a$ and $c$ with $y$. 
		By comparing these expressions with the entries of Table~\ref{tab:assoc.P3}, 
		we conclude that $r_i = s_i$ for all $i=7,\dots,17$.

		Therefore, the product defined on $\mathbb{P}_3$ is associative. 
		Moreover, the neutral element is the zero vector of $\mathbb{P}_3$ 
		and, by a direct computation, the inverse of an element $\overline{\alpha}$ is obtained as follows. 

		Let 
		$
		\overline{\alpha} = (\alpha_1,\ldots,\alpha_6;\;
		\alpha_{12}^1,\alpha_{13}^1,\alpha_{23}^1;\;
		\alpha_{12}^2,\alpha_{13}^2,\alpha_{23}^2;\;
		\alpha_{12}^3,\alpha_{13}^3,\alpha_{23}^3;\;
		\alpha_{123}^1,\alpha_{123}^2),
		$
		and let its inverse be 
		$
		\overline{\beta} = -\overline{\alpha} = (\beta_1,\ldots,\beta_6;\;
		\beta_{12}^1,\beta_{13}^1,\beta_{23}^1;\;
		\beta_{12}^2,\beta_{13}^2,\beta_{23}^2;\;
		\beta_{12}^3,\beta_{13}^3,\beta_{23}^3;\;
		\beta_{123}^1,\beta_{123}^2).
		$
		Recall that the last $11$ coordinates are taken modulo $2$.
		Each coordinate of the inverse of $\overline{\alpha}$ is given by: 
		\begin{equation*}\label{eq:inverse-P3}
			\begin{aligned}
				& \beta_i = -\alpha_i \quad (i=1,\dots,6), \\[4pt]
				& \beta_{12}^{1} = \alpha_{12}^{1} + \alpha_1\alpha_2,\quad
				\beta_{13}^{1} = \alpha_{13}^{1} + \alpha_1\alpha_3,\quad
				\beta_{23}^{1} = \alpha_{23}^{1} + \alpha_2\alpha_3, \\
				& \beta_{12}^{2} = \alpha_{12}^{2} + \alpha_4\alpha_5,\quad
				\beta_{13}^{2} = \alpha_{13}^{2} + \alpha_4\alpha_6,\quad
				\beta_{23}^{2} = \alpha_{23}^{2} + \alpha_5\alpha_6, \\
				& \beta_{12}^{3} = \alpha_{12}^{3} + \alpha_4\alpha_2 + \alpha_5\alpha_1,\quad
				\beta_{13}^{3} = \alpha_{13}^{3} + \alpha_4\alpha_3 + \alpha_6\alpha_1,\quad
				\beta_{23}^{3} = \alpha_{23}^{3} + \alpha_5\alpha_3 + \alpha_6\alpha_2, \\[4pt]
				& \beta_{123}^{1} = \alpha_{123}^{1}
				+ \alpha_{12}^{3}\alpha_3 + \alpha_{13}^{3}\alpha_2 + \alpha_{23}^{3}\alpha_1
				+ \alpha_4\alpha_2\alpha_3 + \alpha_5\alpha_1\alpha_3 + \alpha_6\alpha_1\alpha_2, \\[2pt]
				& \beta_{123}^{2} = \alpha_{123}^{2}
				+ \alpha_{12}^{3}\alpha_6 + \alpha_{13}^{3}\alpha_5 + \alpha_{23}^{3}\alpha_4
				+ \alpha_4\alpha_5\alpha_3 + \alpha_4\alpha_6\alpha_2 + \alpha_5\alpha_6\alpha_1 \\
				& \hspace{3.8em}
				+ \alpha_4(\alpha_2\alpha_6 + \alpha_5\alpha_3)
				+ \alpha_5(\alpha_1\alpha_6 + \alpha_4\alpha_3)
				+ \alpha_6(\alpha_1\alpha_5 + \alpha_4\alpha_2).
			\end{aligned}
		\end{equation*}
		It follows from the above computations that $\mathbb{P}_3$ is a group. \qedhere
		
		\begin{table}[H]
			\footnotesize 
			\begin{tabular}{c||l|l}
				$i$ & $r_i$ & $r_i$ \\ \hline \hline 
				7   & $x_7 + c_7 + x_2 c_1$ & $a_7 + b_7 + c_7 + a_2 b_1 + (a_2 + b_2)c_1$ \\
				8   & $x_8 + c_8 + x_3 c_1$ & $a_8 + b_8 + c_8 + a_3 b_1 + (a_3 + b_3)c_1$ \\
				9   & $x_9 + c_9 + x_3 c_2$ & $a_9 + b_9 + c_9 + a_3 b_2 + (a_3 + b_3)c_2$ \\
				10  & $x_{10} + c_{10} + x_5 c_4$ & $a_{10} + b_{10} + c_{10} + a_5 b_4 + (a_5 + b_5)c_4$ \\
				11  & $x_{11} + c_{11} + x_6 c_4$ & $a_{11} + b_{11} + c_{11} + a_6 b_4 + (a_6 + b_6)c_4$ \\
				12  & $x_{12} + c_{12} + x_6 c_5$ & $a_{12} + b_{12} + c_{12} + a_6 b_5 + (a_6 + b_6)c_5$ \\
				13  & $x_{13} + c_{13} + x_4 c_2 + x_5 c_1$ &
				$a_{13} + b_{13} + c_{13} + a_4 b_2 + a_5 b_1 + (a_4 + b_4)c_2 + (a_5 + b_5)c_1$\\
				14  & $x_{14} + c_{14} + x_4 c_3 + x_6 c_1$ &
				$a_{14} + b_{14} + c_{14} + a_4 b_3 + a_6 b_1 + (a_4 + b_4)c_3 + (a_6 + b_6)c_1$\\
				15  & $x_{15} + c_{15} + x_5 c_3 + x_6 c_2$ &
				$a_{15} + b_{15} + c_{15} + a_5 b_3 + a_6 b_2 + (a_5 + b_5)c_3 + (a_6 + b_6)c_2$\\ 
				16  & \begin{tabular}[t]{@{}l@{}}$x_{16} + c_{16} + x_{13}c_3 + x_{14}c_2$ \\ 
					$+ x_{15}c_1 + x_4 c_2 c_3 + x_5 c_1 c_3 + x_6 c_1 c_2$\end{tabular} &
				\begin{tabular}[t]{@{}l@{}}$a_{16} + b_{16} + c_{16} + a_{13}b_3 + a_{14}b_2 + a_{15}b_1 + a_4 b_2 b_3$ \\ 
					$+ a_5 b_1 b_3 + a_6 b_1 b_2 + (a_{13} + b_{13} + a_4 b_2 + a_5 b_1)c_3$ \\ 
					$+ (a_4 + b_4)c_2 c_3 + (a_{14} + b_{14} + a_4 b_3 + a_6 b_1)c_2$ \\ 
					$+ (a_5 + b_5)c_1 c_3 + (a_6 + b_6)c_1 c_2 + (a_{15} + b_{15} + a_5 b_3 + a_6 b_2)c_1$\end{tabular} \\
				17  & \begin{tabular}[t]{@{}l@{}}$x_{17} + c_{17} + x_{13}c_6 + x_{14}c_5$ \\ 
					$+ x_{15}c_4 + x_4 x_5 c_3 + x_4 x_6 c_2 + x_5 x_6 c_1$ \\ 
					$+ x_{4}(c_2 c_6 + c_5 c_3) + x_{5}(c_1 c_6 + c_4 c_3)$ \\ 
					$+ x_{6}(c_1 c_5 + c_4 c_2)$\end{tabular} &
				\begin{tabular}[t]{@{}l@{}}$a_{17} + b_{17} + c_{17} + a_{13}b_6 + a_{14}b_5 + a_{15}b_4 + a_4 a_5 b_3 + a_4 a_6 b_2 $ \\ 
					$ + a_5 a_6 b_1 + a_{4}(b_2 b_6 + b_5 b_3) + (a_6 + b_6)(c_1 c_5 + c_4 c_2)$ \\ 
					$+ a_{5}(b_1 b_6 + b_4 b_3) + (a_{14} + b_{14} + a_4 b_3 + a_6 b_1)c_5$ \\ 
					$+ (a_{13} + b_{13} + a_4 b_2 + a_5 b_1)c_6 + a_{6}(b_1 b_5 + b_4 b_2)$ \\ 
					$+ (a_{15} + b_{15} + a_5 b_3 + a_6 b_2)c_4 + (a_4 + b_4)(a_5 + b_5)c_3$ \\ 
					$+ (a_4 + b_4)(a_6 + b_6)c_2 + (a_5 + b_5)(a_6 + b_6)c_1$ \\ 
					$+ (a_4 + b_4)(c_2 c_6 + c_5 c_3) + (a_5 + b_5)(c_1 c_6 + c_4 c_3)$\end{tabular}
			\end{tabular}
			\caption{Explicit computation of $r_i$ for $i=7,\dots,17$}
			\label{tab:assoc.P3}
		\end{table}
	\end{proof}
	For the general case, similar computations show that $\mathbb{P}_n$ is a group. 
	For reference, the product of elements in $\mathbb{P}_n$ is given in Table~\ref{tab:mult.pn} (see Appendix~\ref{appendice:tables}), 
	and the corresponding computations for associativity are summarized in Tables~\ref{tab:assoc.pn.parte1}--\ref{tab:assoc.pn.parte2}.

	\begin{proposition}
		$\mathbb{P}_n$ is a group.  
	\end{proposition}
	

	Since $\mathbb{P}_n$ is a group whose structure reflects the relations of $G_n$, we can naturally establish an isomorphism between $\mathbb{P}_n$ and $G_n$, as given in the following proposition.

	\begin{proposition}\label{prop.iso}
		The map $\varphi \colon \mathbb{P}_n \to G_n$ defined by
		\begin{equation*}
			\begin{aligned}
				\varphi(\overline{\alpha}) \;=\;& 
				a_1^{\alpha_1}\cdots a_n^{\alpha_n}\,
				b_1^{\alpha_{1+n}}\cdots b_n^{\alpha_{2n}} \,
				u_{12}^{\alpha_{12}^1}\cdots u_{n-1,\,n}^{\alpha_{n-1,\,n}^1}
				\, v_{12}^{\alpha_{12}^2}\cdots v_{n-1,\,n}^{\alpha_{n-1,\,n}^2} \\[4pt]
				& p_{12}^{\alpha_{12}^3}\cdots p_{n-1,\,n}^{\alpha_{n-1,\,n}^3} \,
				z_{123}^{\alpha_{123}^1}\cdots z_{n-2,\,n-1,\,n}^{\alpha_{n-2,\,n-1,\,n}^1} \,
				t_{123}^{\alpha_{123}^2}\cdots t_{n-2,\,n-1,\,n}^{\alpha_{n-2,\,n-1,\,n}^2}
			\end{aligned}
		\end{equation*}
		is an isomorphism.
	\end{proposition}


	The groups $\mathbb{P}_n$ and $G_n$ are isomorphic, then we can use this isomorphism to prove that $G_n$ is a group with triality. In the next subsection, for the case $n = 3$, we introduce automorphisms $\sigma$ and $\rho$ of $G_3$ of order $2$ and $3$, respectively, in order to prove that $G_3$ is a group with triality $S = \langle \sigma, \rho \ \mid \ \sigma^2 = \rho^3 = (\sigma \rho)^2 = 1 \rangle \cong S_3$. This result will later be extended to establish the general case for $G_n$ (see Subsection~\ref{subsec:general.case}).
	

	
	\subsection{Automorphisms of $G_3$}
	Throughout this subsection, we consider $G = G_3$, a nilpotent group of class $3$ generated by $\left\{a_1,a_2, a_3, b_1, b_2, b_3 \right\}$ and satisfying the relations \ref{rel.1} through \ref{rel.8}. 
	We suppose $u_{ij} = [a_i, a_j],$ $v_{ij} = [b_i, b_j]$, $p_{ij} = [a_i, b_j]$ for $i < j$ and $z_{123}$, $t_{123}$ defined accordingly by these relations. Let the maps $\sigma $, $\tau$ of $G$ be defined by:
	
	\begin{equation}\label{eq.sigma}
		a_{i}^{\sigma} = b_{i}, b_{i}^{\sigma} = a_i, u_{ij}^{\sigma} = v_{ij}, v_{ij}^{\sigma} = u_{ij}, p_{ij}^\sigma = p_{ij}, z_{ijk}^\sigma = t_{ijk}, t_{ijk}^\sigma = z_{ijk}.  
	\end{equation}
	
	\begin{equation}\label{eq.tau}
		a_i^\tau = a_i, b_i^\tau = a_{i}^{-1}b_{i}^{-1}, u_{ij}^\tau = u_{ij}, v_{ij}^\tau = u_{ij}v_{ij}, p_{ij}^\tau = u_{ij}p_{ij}, z_{ijk}^{\tau} = z_{ijk}, t_{ijk}^{\tau} = z_{ijk}t_{ijk}.  
	\end{equation}
	
	
	Next, we first prove that $\sigma$ and $\tau$ are automorphismos of $G$ of order $2$. Then, we will construct another automorphism from those of order $3$ (see Remark \ref{remark:rho}).

	\begin{lemma}\label{lemma.sigma}
		Let  $\sigma$ be defined as above and write $x \in G$ in the following form: 
		\begin{equation} \label{eq.normalform2}
			x = a_{1}^{\alpha_1}a_{2}^{\alpha_2}a_{3}^{\alpha_3} b_{1}^{\alpha_4}b_{2}^{\alpha_5}b_{3}^{\alpha_6} u_{12}^{\alpha_{7}}u_{13}^{\alpha_{8}}u_{23}^{\alpha_{9}
			}v_{12}^{\alpha_{10}}v_{13}^{\alpha_{11}}v_{23}^{\alpha_{12}}p_{12}^{\alpha_{13}}p_{13}^{\alpha_{14}}p_{23}^{\alpha_{15}} z_{123}^{\alpha_{16}} t_{123}^{\alpha_{17}},
		\end{equation}
		where $\alpha_l \in \mathbb{Z}$ for $l=1,\dots,6$ and $\alpha_l \in \mathbb{F}_2$ for $l=7,\dots,17$.
		Then:
		\begin{equation} 
			x^\sigma = a_{1}^{\alpha_4}a_{2}^{\alpha_5}a_{3}^{\alpha_6} b_{1}^{\alpha_1}b_{2}^{\alpha_2}b_{3}^{\alpha_3} u_{12}^{\alpha_{10}}u_{13}^{\alpha_{11}}u_{23}^{\alpha_{12}
			}v_{12}^{\alpha_{7}}v_{13}^{\alpha_{8}}v_{23}^{\alpha_{9}}p_{12}^{\alpha_{13}^{'}}p_{13}^{\alpha_{14}^{'}}p_{23}^{\alpha_{15}^{'}} z_{123}^{\alpha_{16}^{'}} t_{123}^{\alpha_{17}^{'}},
		\end{equation}
		where\\ $\alpha_{13}^{'} = \alpha_{13} + \alpha_{1}\alpha_{5} + \alpha_{2}\alpha_{4}$, $\alpha_{14}^{'} = \alpha_{14} + \alpha_{1}\alpha_{6} + \alpha_{3}\alpha_{4}$, $\alpha_{15}^{'} = \alpha_{15} + \alpha_{2}\alpha_{6} + \alpha_{3}\alpha_{5}$, \linebreak $\alpha_{16}^{'} = \alpha_{17} + \alpha_{1}\alpha_{5}\alpha_{6} + \alpha_{2}\alpha_{4}\alpha_{6} + \alpha_{3}\alpha_{4}\alpha_{5}$ and $\alpha_{17}^{'} = \alpha_{16} + \alpha_{1}\alpha_{2}\alpha_{6} + \alpha_{1}\alpha_{3}\alpha_{5} + \alpha_{2}\alpha_{3}\alpha_{4}$.
	\end{lemma}
	
	\begin{proof}
		
		Let $x\in G$ defined by Equation \ref{eq.normalform2}. Applying $\sigma$ we get: 
		
		\begin{equation} 
			x^\sigma = b_{1}^{\alpha_1}b_{2}^{\alpha_2}b_{3}^{\alpha_3} a_{1}^{\alpha_4}a_{2}^{\alpha_5}a_{3}^{\alpha_6} v_{12}^{\alpha_{7}}v_{13}^{\alpha_{8}}v_{23}^{\alpha_{9}
			}u_{12}^{\alpha_{10}}u_{13}^{\alpha_{11}}u_{23}^{\alpha_{12}}p_{12}^{\alpha_{13}}p_{13}^{\alpha_{14}}p_{23}^{\alpha_{15}} t_{123}^{\alpha_{16}} z_{123}^{\alpha_{17}},
		\end{equation}
		where $\alpha_l \in \mathbb{Z}$, $l=1,\dots,6$ and $\alpha_l \in \mathbb{Z}_2$, $l=7,\dots,17$.
		
		We need to write $x^\sigma$ in the form given by Remark \ref{rem.normalform}.  Consider $g_1 = a_{1}^0 a_{2}^{0}a_{3}^{0} b_{1}^{\alpha_1}b_{2}^{\alpha_2}b_{3}^{\alpha_3}$ and $g_2 = a_{1}^{\alpha_4}a_{2}^{\alpha_5}a_{3}^{\alpha_6} b_{1}^{0}b_{2}^{0}b_{3}^{0}p_{12}^{\alpha_{13}}p_{13}^{\alpha_{14}}p_{23}^{\alpha_{15}}$. Applying the Lemma \ref{lemma.product}, we obtain the desire result.
	\end{proof}
	
	As a consequence, using the isomorphism $G_3 \cong \mathbb{P}_3$, 
	the action of $\sigma$ can be described in terms of the exponents as follows.
	
	\begin{remark}\label{rem.sigma}
		Let $\sigma$ be given by Definition \ref{eq.sigma}. If $\overline{\alpha} = (\alpha_1, \ldots, \alpha_{17}) \in \mathbb{P}_3$, then as a direct consequence of the Lemma \ref{lemma.sigma} we can write 
		\begin{equation}\label{eq.sigma2}
			\overline{\alpha}^{\sigma} = (\alpha_{4}, \alpha_{5}, \alpha_{6},\alpha_{1},\alpha_{2},\alpha_{3},\alpha_{10},\alpha_{11},\alpha_{12},\alpha_{7},\alpha_{8},\alpha_{9},\alpha_{13}^{'},\alpha_{14}^{'},\alpha_{15}^{'},\alpha_{16}^{'},\alpha_{17}^{'}),
		\end{equation}
		where $\alpha_{l}^{'}$ for $l=13,\dots,17$, are obtained as in Lemma \ref{lemma.sigma}.
	\end{remark}
	
	From the previous description, we can now prove that $\sigma$ defines an automorphism of $G$ of order $2$.
	
	\begin{proposition}\label{prop.sigma}
		The map $\sigma$ given by Definition $\ref{eq.sigma}$ is an automorphism of $G$ of order $2$.
	\end{proposition}
	\begin{proof}
		Let $\overline{x}, \overline{y}, \overline{z} = \overline{x}\cdot \overline{y} \in \mathbb{P}_{3} \cong G$ be written in the form $\overline{x} = (x_1, \ldots,x_{17})$, $\overline{y} = (y_1, \ldots,y_{17})$ and $\overline{z} = (z_1, \ldots,z_{17})$ , respectively. By Remark \ref{rem.sigma}, we obtain the images of these elements by $\sigma$. We suppose that $\overline{x}^\sigma = (x_1^{'}, \ldots,x_{17}^{'})$, $\overline{y}^\sigma = (y_1^{'}, \ldots,y_{17}^{'})$ and $\overline{z}^\sigma = (z_1^{'}, \ldots,z_{17}^{'})$ , respectively. 
		
		Using the product defined in Table \ref{tab:mult.p3}, we determine the products $\overline{x}\cdot \overline{y}$ and $\overline{x}^\sigma \cdot \overline{y}^\sigma$. We denote $\overline{w} = \overline{x}^\sigma \cdot \overline{y}^\sigma = (w_1, \ldots,w_{17})$. First, we need to prove that $z_i^{'} = w_i$, for all $i=1,\ldots,17$. Indeed:
		\begin{eqnarray*}
			z_1^{'} &=& z_4 = x_4 + y_4 = x_1^{'} + y_1^{'} = w_1;\\
			z_7^{'} &=& z_{10} = x_{10} + y_{10} + x_5 y_4 = x_7^{'} + y_7^{'} + x_2^{'} y_1^{'} = w_7; \\
			z_{13}^{'} &=&  z_{13} + z_1 z_5 + z_2 z_4 = x_{13} + y_{13} + \cancel{x_4 y_2} + \cancel{x_5 y_1} + x_1 x_5 + x_1 y_5 + \cancel{x_5 y_1} + y_1 y_5 +\\
			&\,& + x_2 x_4 + x_2 y_4 + \cancel{x_4 y_2} + y_2 y_4 = x_{13}^{'} + y_{13}^{'} + x_{4}^{'}y_{2}^{'} + x_{5}^{'}y_{1}^{'} = w_{13}; \\
			z_{16}^{'} &=&  z_{17} + z_1 z_5 z_6 + z_2 z_4 z_6 + z_3 z_4 z_5 = 
			x_{17} + y_{17} + x_{13}y_6 + x_{14}y_5 + x_{15}y_4  + \cancel{x_4 x_5 y_3} + \\
			&\,& + \cancel{x_4 x_6 y_2} + \cancel{x_5 x_6 y_1} + \cancel{x_4 y_2 y_6} + \cancel{x_4 y_3 y_5} + \cancel{x_5 y_1 y_6} + \cancel{x_5 y_3 y_4} + \cancel{x_6 y_1 y_5} + \cancel{x_6 y_2 y_4} +\\
			&\,& + (x_1 x_5 x_6 + x_1 x_6 y_5 + \cancel{x_5 x_6 y_1} + \cancel{x_6 y_1 y_5} + x_1 x_5 y_6 + x_1 y_5 y_6 + \cancel{x_5 y_1 y_6} + y_1 y_5 y_6) +\\ 
			&\,& + (x_2 x_4 x_6 + x_2 x_6 y_4  + \cancel{x_4 x_6 y_2} + \cancel{x_6 y_2 y_4} + x_2 x_4 y_6 + x_2 y_4 y_6 + \cancel{x_4 y_2 y_6} +  y_2 y_4 y_6) +\\ 
			&\,& + (x_3 x_4 x_5 + x_3 x_5 y_4  + \cancel{x_4 x_5 y_3} + \cancel{x_5 y_3 y_4} + x_3 x_4 y_5 + x_3 y_4 y_5 + \cancel{x_4 y_3 y_5} +  y_3 y_4 y_5) =\\ 
			&=& (x_{17} + x_1 x_5 x_6 + x_2 x_4 x_6 + x_3 x_4 x_5) + (y_{17} + y_1 y_5 y_6 + y_2 y_4 y_6 + y_3 y_4 y_5) + \\
			&\,& + (x_{13} + x_1 x_5 + x_2 x_4)y_6 + (x_{14} + x_1 x_6 + x_3 x_4)y_5 + (x_{15} + x_2 x_6 + x_3 x_5)y_4 + \\
			&\,& + x_1 y_5 y_6 + x_2 y_4 y_6 + x_3 y_4 y_5 = x_{16}^{'} + y_{16}^{'} + x_{13}^{'}y_{3}^{'} + x_{14}^{'}y_{2}^{'} + x_{15}^{'}y_{1}^{'} + x_{4}^{'}y_{2}^{'}y_{3}^{'} + \\
			&\,& + x_{5}^{'}y_{1}^{'}y_{3}^{'} + x_{6}^{'}y_{1}^{'}y_{2}^{'} = w_{16}. 
		\end{eqnarray*}
		The other terms are obtained in a similar way. At last, by Equation \ref{eq.sigma2}, it is trivial that $\overline{x}^{\sigma^2} = \overline{x}$. Therefore, $\sigma$ is an automorphism of $G$ of order $2$.
	\end{proof}
	
	
	In analogy with the previous case, we now compute the action of $\tau$ on an element of $G$ written in normal form.
	
	
	\begin{lemma}\label{lemma.tau}
		Let  $\tau$ be as defined in Definition $\ref{eq.tau}$ and $x \in G$ be written as in  $\ref{eq.normalform2}$. 
		Then:
		\begin{equation} \label{}
			x^\tau = a_{1}^{\alpha_1 - \alpha_4}a_{2}^{\alpha_2 - \alpha_5}a_{3}^{\alpha_3 - \alpha_6} b_{1}^{-\alpha_4}b_{2}^{-\alpha_5}b_{3}^{-\alpha_6} u_{12}^{\alpha_{7}^{'}}u_{13}^{\alpha_{8}^{'}}u_{23}^{\alpha_{9}^{'}
			}v_{12}^{\alpha_{10}}v_{13}^{\alpha_{11}}v_{23}^{\alpha_{12}}p_{12}^{\alpha_{13}^{'}}p_{13}^{\alpha_{14}^{'}}p_{23}^{\alpha_{15}^{'}} z_{123}^{\alpha_{16}^{'}} t_{123}^{\alpha_{17}},
		\end{equation}
		where \\
		$\alpha_{7}^{'} = \alpha_{7} + \alpha_{10} + \alpha_{13} + \alpha_{2}\alpha_{4}$, $\alpha_{8}^{'} = \alpha_{8} + \alpha_{11} + \alpha_{14} + \alpha_{3}\alpha_{4}$, $\alpha_{9}^{'} = \alpha_{9} + \alpha_{12} + \alpha_{15} + \alpha_{3}\alpha_{5}$, \linebreak
		$\alpha_{13}^{'} = \alpha_{13} + \alpha_{4}\alpha_{5} $, $\alpha_{14}^{'} = \alpha_{14} + \alpha_{4}\alpha_{6}$, $\alpha_{15}^{'} = \alpha_{15} + \alpha_{5}\alpha_{6}$ and  $\alpha_{16}^{'} = \alpha_{16} + \alpha_{17} + \alpha_{4}\alpha_{5}\alpha_{6} $.
	\end{lemma}
	
	\begin{proof}
		Consider $x\in G$ as described in the statement and apply $\tau$ to obtain the following result:
		\begin{equation} \label{eq.tau2}
			x^\tau = a_{1}^{\alpha_1}a_{2}^{\alpha_2}a_{3}^{\alpha_3} (a_{1}^{-\alpha_4} b_{1}^{-\alpha_4})(a_{2}^{-\alpha_5} b_{2}^{-\alpha_5})(a_{3}^{-\alpha_6} b_{3}^{-\alpha_6})\mu,
		\end{equation}
		where\\
		\[\mu = u_{12}^{\alpha_{7}}u_{13}^{\alpha_{8}}u_{23}^{\alpha_{9}
		}(u_{12}^{\alpha_{10}}v_{12}^{\alpha_{10}})(u_{13}^{\alpha_{11}}v_{13}^{\alpha_{11}})(u_{23}^{\alpha_{12}}v_{23}^{\alpha_{12}})(u_{12}^{\alpha_{13}} p_{12}^{\alpha_{13}})(u_{13}^{\alpha_{14}} p_{13}^{\alpha_{14}})(u_{23}^{\alpha_{15}} p_{23}^{\alpha_{15}}) z_{123}^{\alpha_{16}} (z_{123}^{\alpha_{17}} t_{123}^{\alpha_{17}}). \]
		
		We denote $g_1 = a_{1}^{\alpha_1}a_{2}^{\alpha_2}a_{3}^{\alpha_3} a_{1}^{-\alpha_4}$, $g_2 = b_{1}^{-\alpha_4} a_{2}^{-\alpha_5} b_{2}^{-\alpha_5} a_{3}^{-\alpha_6}$, $g_3 = b_{3}^{-\alpha_6}$ and so, we can write $x^\tau$ as
		\[x^\tau = g_1 g_2 g_3 \mu.\]
		
		To obtain the desired result, we will apply the method from Lemma \ref{lemma.product} in $g_1$, $g_2$ and $g_2 g_3$ and thus present the elements in the established order. In fact:
		\begin{enumerate}
			\item $g_1 = a_{1}^{\alpha_1 - \alpha_4}a_{2}^{\alpha_2}a_{3}^{\alpha_3} u_{12}^{\alpha_2 \alpha_4} u_{13}^{\alpha_3 \alpha_4};$
			\item $g_2 = \underbrace{b_{1}^{-\alpha_4} a_{2}^{-\alpha_5}} b_{2}^{-\alpha_5} a_{3}^{-\alpha_6} = a_{2}^{-\alpha_5}  b_{1}^{-\alpha_4}  \underbrace{p_{12}^{\alpha_4 \alpha_5} b_{2}^{-\alpha_5}} a_{3}^{-\alpha_6} = a_{2}^{-\alpha_5}  b_{1}^{-\alpha_4}   b_{2}^{-\alpha_5} \underbrace{p_{12}^{\alpha_4 \alpha_5} a_{3}^{-\alpha_6}} = \\$
			$= a_{2}^{-\alpha_5}  b_{1}^{-\alpha_4}  \underbrace{b_{2}^{-\alpha_5} a_{3}^{-\alpha_6}} p_{12}^{\alpha_4 \alpha_5} z^{\alpha_4 \alpha_5 \alpha_6} = a_{2}^{-\alpha_5} \underbrace{b_{1}^{-\alpha_4}    a_{3}^{-\alpha_6}} b_{2}^{-\alpha_5} p_{12}^{\alpha_4 \alpha_5} p_{23}^{\alpha_5 \alpha_6} z^{\alpha_4 \alpha_5 \alpha_6} =\\  $
			$= a_{2}^{-\alpha_5} a_{3}^{-\alpha_6}  b_{1}^{-\alpha_4} \underbrace{p_{13}^{\alpha_4 \alpha_6} b_{2}^{-\alpha_5}} p_{12}^{\alpha_4 \alpha_5} p_{23}^{\alpha_5 \alpha_6} z^{\alpha_4 \alpha_5 \alpha_6} =\\$
			$= a_{2}^{-\alpha_5} a_{3}^{-\alpha_6}  b_{1}^{-\alpha_4} b_{2}^{-\alpha_5}  p_{12}^{\alpha_4 \alpha_5} p_{13}^{\alpha_4 \alpha_6} p_{23}^{\alpha_5 \alpha_6} z^{\alpha_4 \alpha_5 \alpha_6} t^{\alpha_4 \alpha_5 \alpha_6};$
			\item $g_2 g_3 = a_{2}^{-\alpha_5} a_{3}^{-\alpha_6}  b_{1}^{-\alpha_4} b_{2}^{-\alpha_5} \underbrace{ p_{12}^{\alpha_4 \alpha_5} p_{13}^{\alpha_4 \alpha_6} p_{23}^{\alpha_5 \alpha_6} b_{3}^{-\alpha_6}} z^{\alpha_4 \alpha_5 \alpha_6} t^{\alpha_4 \alpha_5 \alpha_6}=\\ $
			$= a_{2}^{-\alpha_5} a_{3}^{-\alpha_6}  b_{1}^{-\alpha_4} b_{2}^{-\alpha_5} b_{3}^{-\alpha_6} p_{12}^{\alpha_4 \alpha_5} p_{13}^{\alpha_4 \alpha_6} p_{23}^{\alpha_5 \alpha_6}  z^{\alpha_4 \alpha_5 \alpha_6} \cancel{t^{2\alpha_4 \alpha_5 \alpha_6}}.$
			
			Since $G_3$ is nilpotent of class 3 and satisfies the Properties $(1)-(9)$, we get,
			
			\begin{eqnarray*}
				g_1 g_2 g_3 &=&  a_{1}^{\alpha_1 - \alpha_4}a_{2}^{\alpha_2}\underbrace{a_{3}^{\alpha_3} a_{2}^{-\alpha_5}} a_{3}^{-\alpha_6}  b_{1}^{-\alpha_4} b_{2}^{-\alpha_5} b_{3}^{-\alpha_6} p_{12}^{\alpha_4 \alpha_5} p_{13}^{\alpha_4 \alpha_6} p_{23}^{\alpha_5 \alpha_6} u_{12}^{\alpha_2 \alpha_4} u_{13}^{\alpha_3 \alpha_4}z^{\alpha_4 \alpha_5 \alpha_6} \\
				&=& a_{1}^{\alpha_1 - \alpha_4}a_{2}^{\alpha_2} a_{2}^{-\alpha_5}a_{3}^{\alpha_3} a_{3}^{-\alpha_6}  b_{1}^{-\alpha_4} b_{2}^{-\alpha_5} b_{3}^{-\alpha_6} p_{12}^{\alpha_4 \alpha_5} p_{13}^{\alpha_4 \alpha_6} p_{23}^{\alpha_5 \alpha_6} u_{12}^{\alpha_2 \alpha_4} u_{13}^{\alpha_3 \alpha_4}u_{23}^{\alpha_3 \alpha_5} z^{\alpha_4 \alpha_5 \alpha_6}.
			\end{eqnarray*}
			
		\end{enumerate}
		
		Therefore, \[x^\tau = a_{1}^{\alpha_1 - \alpha_4}a_{2}^{\alpha_2 - \alpha_5}a_{3}^{\alpha_3 - \alpha_6} b_{1}^{-\alpha_4}b_{2}^{-\alpha_5}b_{3}^{-\alpha_6} \mu^{'}, \]
		
		where
		\[\mu^{'} = u_{12}^{\alpha_{7}^{'}}u_{13}^{\alpha_{8}^{'}}u_{23}^{\alpha_{9}^{'}
		}v_{12}^{\alpha_{10}}v_{13}^{\alpha_{11}}v_{23}^{\alpha_{12}} p_{12}^{\alpha_{13}^{'}} p_{13}^{\alpha_{14}^{'}} p_{23}^{\alpha_{15}^{'}} z_{123}^{\alpha_{16}^{'}} t_{123}^{\alpha_{17}},\]
		
		and $\alpha_{i}^{'}$ are as we wanted.
	\end{proof}
	
	As a consequence, using the isomorphism $G_3 \cong \mathbb{P}_3$, 
	the action of $\tau$ can be described in terms of the exponents as follows.
	
	\begin{remark}\label{rem.tau}
		Let $\tau$ be given by Definition \ref{eq.tau}, and let $\overline{\alpha} = (\alpha_1, \ldots, \alpha_{17}) \in \mathbb{P}_3$. 
		It follows directly from Proposition \ref{prop.iso} and Lemma \ref{lemma.tau} that:  \begin{eqnarray*}\label{eq.tau2}
			\overline{\alpha}^{\tau} &=& (\alpha_{1} - \alpha_{4}, \alpha_{2} - \alpha_{5}, \alpha_{3} - \alpha_{6},-\alpha_{4},-\alpha_{5},-\alpha_{6},\alpha_{7} + \alpha_{10} + \alpha_{13} + \alpha_{2}\alpha_{4},\\ &\;& \alpha_{8} + \alpha_{11} + \alpha_{14} + \alpha_{3}\alpha_{4},  \alpha_{9} + \alpha_{12} + \alpha_{15} + \alpha_{3}\alpha_{5}, \alpha_{10},\alpha_{11},\alpha_{12}, \alpha_{13} + \alpha_{4}\alpha_{5},\\ 
			&\;&\alpha_{14} + \alpha_{4}\alpha_{6},\alpha_{15} + \alpha_{5}\alpha_{6},\alpha_{16} + \alpha_{17} + \alpha_{4}\alpha_{5}\alpha_{6},\alpha_{17}).
		\end{eqnarray*}
	\end{remark}
	
	Using $\tau$ as an auxiliary step, we define $\rho = \sigma \circ \tau$, 
	which will play the role of the second generator in the $S_3$-action. 
	
	\begin{remark}\label{remark:rho}
		Let $\sigma$ and $\tau$ as above. Note that, if we consider $\rho = \sigma \circ\tau$, we can express $\rho$ in the following way:
		
		\begin{equation}\label{eq.rho}
			a_{i}^{\rho} = b_{i}, b_{i}^{\rho} = a_{i}^{-1}b_{i}^{-1}, u_{ij}^{\rho} = v_{ij}, v_{ij}^{\rho} = u_{ij}v_{ij}, p_{ij}^\rho =v_{ij} p_{ij}, z_{ijk}^\rho = t_{ijk}, t_{ijk}^\rho = z_{ijk}t_{ijk}.  
		\end{equation}
	\end{remark}
	
	With $\rho$ defined in terms of $\sigma$ and $\tau$, the following proposition summarizes their main properties.
	
	\begin{proposition}\label{prop.rho}
		Let the maps $\sigma$ and $\tau$ be as defined in Definitions $\ref{eq.sigma}$ and $\ref{eq.tau}$, respectively. Then $\tau$ and $\rho = \sigma \circ\tau$ are automorphisms of $G$ of order $2$ and $3$, respectively.
	\end{proposition}
	\begin{proof}
		Let $\overline{x}, \overline{y}, \overline{z} = \overline{x}\cdot \overline{y} \in \mathbb{P}_{3} \cong G$ be expressed as $\overline{x} = (x_1, \ldots,x_{17})$, $\overline{y} = (y_1, \ldots,y_{17})$ and $\overline{z} = (z_1, \ldots,z_{17})$ , respectively. According to Remark \ref{rem.tau}, we determine the images of these elements under $\tau$. We assume that $\overline{x}^\tau = (x_1^{'}, \ldots,x_{17}^{'})$, $\overline{y}^\tau = (y_1^{'}, \ldots,y_{17}^{'})$ and $\overline{z}^\tau = (z_1^{'}, \ldots,z_{17}^{'})$ , respectively. 
		
		By referring to the Table \ref{tab:mult.p3}, we calculate the products $\overline{x}\cdot \overline{y}$ and $\overline{x}^\tau \cdot \overline{y}^\tau$. We denote $\overline{w} = \overline{x}^\tau \cdot \overline{y}^\tau = (w_1, \ldots,w_{17})$. We need to prove that $z_i^{'} = w_i$, for all $i=1,\ldots,17$ to show that $\tau$ is a homomorphism of $G_3$. Below, we provide the calculations for some of the terms; the remaining terms can be obtained in a similar manner.
		\begin{eqnarray*}
			z_1^{'} &=& z_1 - z_4 = (x_1 + y_1) - (x_4 + y_4) = x_1^{'} + y_1^{'} = w_1;\\
			z_4^{'} &=& -z_4 = -x_4 - y_4 = x_4^{'} + y_4^{'} = w_4; \\
			z_7^{'} &=& z_{7} + z_{10} + z_{13} + z_2 z_4 = x_7 + y_7 + x_{10} + y_{10} + x_{13} + y_{13}  + x_2 y_1+ \cancel{x_4 y_2} + \\
			&\,& + x_5 y_1  + x_5 y_4 + x_2 x_4 + x_2 y_4 + \cancel{x_4 y_2} + y_2 y_4= x_7^{'} + y_7^{'} + x_2^{'} y_1^{'}= w_7;\\
			z_{14}^{'} &=&  z_{14} + z_4 z_6  = x_{14} + y_{14}  + x_4 y_3 + x_6 y_1 + (x_4 + y_4)(x_6 + y_6) =\\
			&=&  x_{14}^{'} + y_{14}^{'} + x_{4}^{'}y_{3}^{'} + x_{6}^{'}y_{1}^{'} = w_{14}; \\
			z_{16}^{'} &=& z_{16} + z_{17} + z_4 z_5 z_6  = x_{16} + y_{16} + x_{17} + y_{17} + x_{13}y_3 + x_{14}y_2 + x_{15}y_1  + \\
			&\,& + x_4 y_2 y_3 + x_5 y_1 y_3 + x_6 y_1 y_2 + x_{13} y_6 + x_{14} y_5 + x_{15} y_4 + x_4 x_5 y_3 + x_4 x_6 y_2 + \\
			&\,& + x_5 x_6 y_1  + x_4(y_2 y_6 + y_5 y_3) + x_5(y_1 y_6 + y_4 y_3) +  x_6(y_1 y_5 + y_4 y_2) + \\ 
			&\,& +(x_4 + y_4)(x_5 + y_5)(x_6 + y_6) =   \\
			&=& x_{16}^{'} + y_{16}^{'} + x_{13}^{'}y_{3}^{'} + x_{14}^{'}y_{2}^{'} + x_{15}^{'}y_{1}^{'} + x_{4}^{'}y_{2}^{'}y_{3}^{'}  + x_{5}^{'}y_{1}^{'}y_{3}^{'} + x_{6}^{'}y_{1}^{'}y_{2}^{'} = w_{16}. 
		\end{eqnarray*}
		
		Moreover,  by Remark \ref{rem.tau}, we have
		
		\vspace{-0.3in}
		
		\begin{eqnarray*}\label{eq.tau2}
			\overline{x}^{\tau^2} &=& 
			(x_1 - x_4 + x_4, x_2 - x_5 + x_5, x_3 - x_6 + x_6, x_4, x_5, x_6, \\
			&\,&x_7 + x_{13}+ x_2 x_4 + (x_2 - x_5)(-x_4) + x_{13} + x_4 x_5, \\
			&\,&x_8 + x_{14}+ x_3 x_4 +  x_{14}   + (x_3 - x_6)(-x_4)  + x_4 x_6, \\
			&\,&x_9 + x_{15}+ x_3 x_5 + (x_3 - x_6)(-x_5) + x_{15} + x_5 x_6,\\
			&\,& x_{10}, x_{11}, x_{12}, x_{13} + x_4 x_5 + x_4 x_5, x_{14} + x_4 x_6 + x_4 x_6,\\
			&\,& x_{15} + x_5 x_6 + x_5 x_6, x_{16}  + x_{17} + x_4 x_5 x_6  + x_{17} + x_4 x_5 x_6, x_{17} ) = \overline{x}.
		\end{eqnarray*}
		\vspace{-0.25in}
		
		Thus $\overline{x}^{\tau^2} = \overline{x}$. Therefore, $\tau$ is an automorphism of $G$ with order $2$.
		
		It is clear that $\rho$ is a homomorphism of $G_3$. Now, to complete the proof, we apply the properties established in Remarks \ref{rem.sigma} and \ref{rem.tau} to explicitly determine $\rho$ and $\rho^2$, and verify that $\overline{x}^{\rho^3} = \overline{x}$, where $\overline{x}$ is the element previously defined.  
		
		This follows by considering $\overline{x}$ and its images  under $\rho$ and $\rho^2$, given by $\overline{x}^\rho = (x_1^{1}, \ldots,x_{17}^{1})$, and $\overline{x}^{\rho^2} = (x_1^{2}, \ldots,x_{17}^{2})$. We now proceed with the explicit computations:
		\vspace{-0.3in}
		
		\begin{eqnarray*}
			\overline{x}^\rho &=& \sigma(\tau(\overline{x})) = \sigma(x_{1} - x_{4}, x_{2} - x_{5}, x_{3} - x_{6},-x_{4},-x_{5},-x_{6},x_{7} + x_{10} + x_{13} + x_{2}x_{4},\\ &\;& x_{8} + x_{11} + x_{14} + x_{3}x_{4},  x_{9} + x_{12} + x_{15} + x_{3}x_{5}, x_{10},x_{11},x_{12},  x_{13} + x_{4}x_{5}, \\ &\;& x_{14} + x_{4}x_{6},x_{15} + x_{5}x_{6},x_{16} + x_{17} + x_{4}x_{5}x_{6},x_{17}) = \\
			&=& (-x_4, -x_5, -x_6, x_1 - x_4, x_2 - x_5, x_3 - x_6, x_{10}, x_{11}, x_{12}, x_7 + x_{10} + x_{13} + x_2 x_4, \\
			&\,& x_8 + x_{11} + x_{14} + x_3 x_4, x_9 + x_{12} + x_{15} + x_3 x_5, x_{13} + x_1 x_5 + x_2 x_4 + x_4 x_5, \\
			&\,& x_{14} + x_1 x_6 + x_3 x_4 + x_4 x_6, x_{15} + x_2 x_6 + x_3 x_5 + x_5 x_6, x_{17} + x_1 x_5 x_6 + x_2 x_4 x_6 + \\
			&\,& + x_3 x_4 x_5 + x_4 x_5 x_6, x_{16} + x_{17} + x_1 x_2 x_6 + x_2 x_3 x_4 + x_1 x_3 x_5 ) = (x_1^{1}, \ldots,x_{17}^{1}).
		\end{eqnarray*}
		
		\begin{eqnarray*}
			\overline{x}^{\rho^2} &=& \rho(x_1^{1}, \ldots,x_{17}^{1})  = (x_4 - x_1, x_5 - x_2, x_6 - x_3, -x_1, -x_2, -x_3, x_7 + x_{10} + x_{13} + \\ 
			&\,& + x_2 x_4, x_8 + x_{11} + x_{14} + x_3 x_4, x_9 + x_{12} + x_{15} + x_3 x_5, x_7, x_8, x_9, x_{13} + x_2 x_4 + \\ 
			&\,& + x_1 x_5 + x_1 x_2, x_{14} + x_3 x_4 + x_1 x_6 + x_1 x_3, x_{15} + x_3 x_5 + x_2 x_6 + x_2 x_3, x_{16} + x_{17} +\\
			&\,& + x_2 x_3 x_4 + x_3 x_4 x_5  + x_1 x_2 x_3 + x_1 x_2 x_6 + x_1 x_3 x_5 + x_1 x_5 x_6 + x_2 x_4 x_6, \\
			&\,& x_{16} + x_1 x_2 x_6 + x_2 x_3 x_4 + x_1 x_3 x_5) = (x_1^{2}, \ldots,x_{17}^{2}).\\
		\\
			\overline{x}^{\rho^3} &=& \rho(x_1^{2}, \ldots,x_{17}^{2}) = (-x_4^{2}, -x_5^{2}, -x_6^{2}, x_1^{2} - x_4^{2}, x_2^{2} - x_5^{2}, x_3^{2} - x_6^{2},  x_{10}^{2}, x_{11}^{2}, x_{12}^{2}, x_7^{2} + \\
			&\,& + x_{10}^{2} + x_{13}^{2} + x_2^{2} x_4^{2},  x_8^{2} + x_{11}^{2} + x_{14}^{2} + x_3^{2} x_4^{2}, x_9^{2} + x_{12}^{2} + x_{15}^{2} + x_3^{2} x_5^{2}, x_{13}^{2} + x_1^{2} x_5^{2} + \\
			&\,& + x_2^{2} x_4^{2} + x_4^{2} x_5^{2},  x_{14}^{2} + x_1^{2} x_6^{2} + x_3^{2} x_4^{2} + x_4^{2} x_6^{2}, x_{15}^{2} + x_2^{2} x_6^{2} + x_3^{2} x_5^{2} + x_5^{2} x_6^{2}, x_{17}^{2} + \\
			&\,& + x_1^{2} x_5^{2} x_6^{2} + x_2^{2} x_4^{2} x_6^{2}  + x_3^{2} x_4^{2} x_5^{2} + x_4^{2} x_5^{2} x_6^{2}, x_{16}^{2} + x_{17}^{2} + x_1^{2} x_2^{2} x_6^{2} + x_2^{2} x_3^{2} x_4^{2} +  \\
			&\,& + x_1^{2} x_3^{2} x_5^{2} ) = (x_1, \ldots,x_{17}) =  \overline{x}.
		\end{eqnarray*}
	\end{proof}
	
	


Using the automorphisms constructed above, we can establish that $G_3$ is a group with triality.

\begin{theorem}
	Consider the automorphisms $\sigma$ and $\rho$ given by Definitions $\ref{eq.sigma}$ and $\ref{eq.rho}$, respectively. Then  $G = G_3$ is a group with triality $S = <\sigma, \rho>$.
\end{theorem}
\begin{proof}
	Let $\sigma$ and $\rho$ as in the statement. By Propositions \ref{prop.sigma} and \ref{prop.rho},   $\sigma$ and $\rho$ are automorphisms of orders $2$ and $3$, respectively. Moreover, by direct computation, we observe that $S = <\sigma,\rho>$ is isomorphic to the symmetric group $S_3$ on three elements. 
	
	In order to prove that $G_3$ is a group with triality $S = \langle \sigma, \rho \rangle$, it remains to prove that for every $x \in G_3$, the Equation 
	
	\begin{equation}\label{eq.triality}
		(x^{-1}x^{\sigma})(x^{-1}x^{\sigma})^{\rho}(x^{-1}x^{\sigma})^{\rho^{2}} = 1
	\end{equation} 
	is satisfied.

	Let \( x \in G_3 \), and let \( x^{-1} \) and \( x^{\sigma} \) be, respectively, the inverse of \( x \) and its image under \( \sigma \). Since, by Proposition \ref{prop.iso}, \( {\mathbb P}_3 \) is isomorphic to \( G_3 \), we can represent these elements by elements of \( {\mathbb P}_3 \). 
	Moreover, as \( \mathbb{P}_3 = \mathbb{Z}^{6} \oplus \mathbb{F}_2^m \) with \( m = 3 {3 \choose 2} + 2 {3 \choose 3} \), we assume that the corresponding elements are, respectively, \( \overline{x} = (x_1,\dots,x_{17}) \), \( \overline{x^{-1}} = \overline{y} = (y_1,\dots,y_{17}) \), and \( \overline{x^\sigma} = (x_{1}^{'},\dots,x_{17}^{'}) \). According to Lemma \ref{lemma:p3group} and Remark \ref{rem.sigma} we have:
	
	\vspace{-0.1in}
	\begin{eqnarray*}
		\overline{y} &=& (-x_1, -x_2, -x_3, -x_4, -x_5,-x_6, x_7 + x_2 x_1, x_8 + x_3 x_1, x_9 + x_3 x_2, x_{10} + x_5 x_4,\\
		&\,&  x_{11} + x_6 x_4, x_{12} + x_6 x_5, x_{13} + x_4 x_2 + x_5 x_1, x_{14} + x_4 x_3 + x_6 x_1, x_{15} + x_5 x_3 + x_6 x_2,\\
		&\,& x_{16} + x_{13} x_3 + x_{14}x_2 + x_{15}x_1 + x_4 x_2 x_3 + x_5 x_1 x_3 + x_6 x_1 x_2, x_{17} + x_{13}x_6 + x_{14}x_5 \\
		&\,& + x_{15}x_4 + x_4 x_2 x_6 + x_5 x_4 x_3 + x_6 x_1 x_5);\\
		\overline{x^\sigma} &=& (x_4, x_5, x_6, x_1, x_2, x_3, x_{10}, x_{11}, x_{12}, x_{7}, x_8, x_9, x_{13} + x_1 x_5 + x_2 x_4, x_{14} + x_1 x_6 + x_3 x_4,\\
		&\,&  x_{15} + x_2 x_6 + x_3 x_5, x_{17} + x_1 x_5 x_6 + x_2 x_4 x_6 + x_3 x_4 x_5, x_{16} + x_1 x_2 x_6 + x_1 x_3 x_5 \\ &\,& + x_2 x_3 x_4).  \\
	\end{eqnarray*}
	
	\vspace{-0.2in}
	

	
	Let \( x, x^{-1} \), and \( x^\sigma \in G_3 \) be as defined above. 
	We define \( a = x^{-1}x^{\sigma} \), \( b = (x^{-1}x^{\sigma})^{\rho} \), and \( c = (x^{-1}x^{\sigma})^{\rho^{2}} \), with their corresponding elements in \( \mathbb{P}_3 \) denoted by \( \overline{a} = (a_1, \dots, a_{17}) \), \( \overline{b} = (b_1, \dots, b_{17}) \), and \( \overline{c} = (c_1, \dots, c_{17}) \).
	In order to explicitly determine these elements, we will use the values of \( \overline{y} \) and \( \overline{x^\sigma} \), applying the definition of multiplication in \( \mathbb{P}_3 \) as well as the automorphisms \( \rho \) and \( \rho^2 \).
	

	Our goal is to prove that \( abc = 1 \), or equivalently, \( \overline{a} \cdot \overline{b} \cdot \overline{c} = 0 \). To do this, it suffices to verify that the coordinates of \( \overline{a} \), \( \overline{b} \), and \( \overline{c} \) satisfy the Linear System \ref{eq:system}, which occurs since each product of three elements is explicitly computed in Table~\ref{tab:assoc.P3}, 
	from which the system is derived.

	\vspace{-0.4in}
	
	\begin{equation}
		\label{eq:system}
		\begin{aligned}
			\begin{cases}
				a_i + b_i + c_i = 0 \,\,\,\, (i = 1,...,6)\\
				a_7 + b_7 + c_7 + a_2 b_1 + (a_2 + b_2)c_1 = 0 \\
				a_8 + b_8 + c_8 + a_3 b_1 + (a_3 + b_3)c_1 = 0 \\
				a_9 + b_9 + c_9 + a_3 b_2 + (a_3 + b_3)c_2 = 0 \\
				a_{10} + b_{10} + c_{10} + a_5 b_4 + (a_5 + b_5)c_4 = 0 \\
				a_{11} + b_{11} + c_{11} + a_6 b_4 + (a_6 + b_6)c_4 = 0 \\
				a_{12} + b_{12} + c_{12} + a_6 b_5 + (a_6 + b_6)c_5 = 0 \\
				a_{13} + b_{13} + c_{13} + a_4 b_2 + a_5 b_1 + (a_4 + b_4)c_2 + (a_5 + b_5)c_1 = 0 \\
				a_{14} + b_{14} + c_{14} + a_4 b_3 + a_6 b_1 + (a_4 + b_4)c_3 + (a_6 + b_6)c_1 = 0 \\
				a_{15} + b_{15} + c_{15} + a_5 b_3 + a_6 b_2 + (a_5 + b_5)c_3 + (a_6 + b_6)c_2 = 0 \\
				a_{16} + b_{16} + c_{16} + a_{13}b_3 + a_{14}b_2 + a_{15}b_1 + a_4 b_2 b_3 + a_5 b_1 b_3 + a_6 b_1 b_2 \\ 
				\quad + (a_{13} + b_{13} + a_4 b_2 + a_5 b_1)c_3 + (a_4 + b_4)c_2 c_3  \\ 
				\quad + (a_{14} + b_{14} + a_4 b_3 + a_6 b_1)c_2 + (a_5 + b_5)c_1 c_3  \\ 
				\quad + (a_{15} + b_{15} + a_5 b_3 + a_6 b_2)c_1 + (a_6 + b_6)c_1 c_2 = 0 \\
				a_{17} + b_{17} + c_{17} + a_{13}b_6 + a_{14}b_5 + a_{15}b_4 + a_4 a_5 b_3 +  a_4 a_6 b_2 + a_5 a_6 b_1 \\ 
				\quad + (a_4 + b_4)(a_5 + b_5)c_3 +(a_4 + b_4)(a_6 + b_6)c_2 + (a_5 + b_5)(a_6 + b_6)c_1  \\ 
				\quad + a_{4}(b_2 b_6 + b_5 b_3) + a_{5}(b_1 b_6 + b_4 b_3) + a_{6}(b_1 b_5 + b_4 b_2)  \\
				\quad + (a_{13} + b_{13} + a_4 b_2 + a_5 b_1)c_6 + (a_4 + b_4)(c_2 c_6 + c_5 c_3) \\
				\quad + (a_{14} + b_{14} + a_4 b_3 + a_6 b_1)c_5 + (a_5 + b_5)(c_1 c_6 + c_4 c_3)  \\
				\quad + (a_{15} + b_{15} + a_5 b_3 + a_6 b_2)c_4 + (a_6 + b_6)(c_1 c_5 + c_4 c_2) = 0 
				
			\end{cases}
		\end{aligned}
	\end{equation}
	\vspace{-0.1in}
	
	Indeed, computing the values of \( \overline{a} \), \( \overline{b} \), and \( \overline{c} \), we obtain:
	
	\vspace{-0.3in}
	
	\begin{eqnarray*}
		\overline{a} &=& (-x_1 + x_4, -x_2 + x_5, -x_3 + x_6, -x_4 + x_1, -x_5 + x_2, -x_6 + x_3,\\
		&\,& x_7 + x_{10} + x_1 x_2 + x_2 x_4, x_8 + x_{11} + x_1 x_3 + x_3 x_4, x_9 + x_{12} + x_2 x_3 + x_3 x_5,\\
		&\,& x_7 + x_{10} + x_4 x_5 + x_1 x_5, x_8 + x_{11} + x_4 x_6 + x_1 x_6, x_9 + x_{12} + x_5 x_6 + x_2 x_6,\\
		&\,& 0, 0, 0, x_{16} + x_{17} + x_3 x_{13} + x_2 x_{14} + x_1 x_2 x_6 + x_1 x_{15} + x_2 x_3 x_4 + x_1 x_3 x_5 +\\
		&\,& +x_6 x_{13} + x_5 x_{14} + x_1 x_5 x_6 + x_4 x_{15} + x_3 x_4 x_5 + x_2 x_4 x_6 + x_4 x_5 x_6,\\
		&\,& x_{16} + x_{17} + x_6 x_{13} + x_5 x_{14} + x_4 x_{15} + x_3 x_{13} + x_2 x_3 x_4 + x_1 x_3 x_5 + x_2 x_{14} + x_1 x_{15}\\
		&\,&  + x_1 x_2 x_6 + x_4 x_5 x_6 + x_3 x_4 x_5 + x_2 x_4 x_6 + x_1 x_5 x_6 \ );\\
	\end{eqnarray*}
	\begin{eqnarray*}
		\overline{b} &=& (x_4 - x_1, x_5 - x_2, x_6 - x_3, -2x_1 + 2x_4, -2x_2 + 2x_5, -2x_3 + 2x_6, \\
		&\,& x_7 + x_{10} + x_4 x_5 + x_1 x_5, x_8 + x_{11} + x_4 x_6 + x_1 x_6, x_9 + x_{12} + x_5 x_6 + x_2 x_6, \\
		&\,& 0, 0, 0, x_4 x_5 + x_2 x_4 + x_1 x_5 + x_1 x_2, x_1 x_6 + x_4 x_6 + x_3 x_4 + x_1 x_3, x_3 x_5 + x_5 x_6 +\\ &\,& + x_2 x_6 + x_2 x_3, x_{16} + x_{17} + x_6 x_{13} + x_5 x_{14} + x_4 x_{15} + x_3 x_{13} + x_2 x_3 x_4 + x_1 x_3 x_5 + \\
		&\,&  + x_2 x_{14} + x_1 x_{15} + x_1 x_2 x_6 + x_4 x_5 x_6 + x_3 x_4 x_5 + x_2 x_4 x_6 + x_1 x_5 x_6, x_1 x_3 x_5 +  \\
		&\,& + x_1 x_2 x_3 + x_1 x_2 x_6 + x_3 x_4 x_5 + x_2 x_3 x_4 + x_2 x_4 x_6);
	\\   
		\overline{c} &=&  ( 2x_1 - 2x_4, 2x_2 - 2 x_5, 2x_3 - 2x_6, x_1 - x_4, x_2 - x_5, x_3 -x_6, 0,0,0,\\
		&\,& x_7 + x_{10} + x_1 x_2 + x_2 x_4, x_8 + x_{11} + x_1 x_3 + x_3 x_4, x_9 + x_{12} + x_2 x_3 + x_3 x_5, x_1 x_2 +\\ 
		&\,& + x_4 x_5 + x_2 x_4 + x_1 x_5, x_1 x_3 + x_4 x_6 + x_3 x_4 + x_1 x_6, x_2 x_3 + x_2 x_6 + x_3 x_5 + x_5 x_6,   \\
		&\,& x_2 x_4 x_6 + x_2 x_3 x_4 + x_1 x_2 x_6 + x_1 x_2 x_3 + x_4 x_5 x_6 + x_3 x_4 x_5 + x_1 x_5 x_6 + x_1 x_3 x_5,\\
		&\,& x_{16} + x_{17} + x_3 x_{13} + x_2 x_{14} + x_1 x_{15} + x_6 x_{13} + x_5 x_{14} + x_4 x_{15} + x_1 x_2 x_3 ).
	\end{eqnarray*}

	With simple calculations, we verify that the coordinates of \( \overline{a} \), \( \overline{b} \), and \( \overline{c} \) satisfy the above system, proving its consistency. 
	Thus, it follows that \( \overline{a} \cdot \overline{b} \cdot \overline{c} = 0 \), which implies \( abc = 1 \) in \( G_3 \). 
	Therefore, the triality condition given by Equation \eqref{eq.triality} holds for all \( x \in G_3 \), concluding that \( G_3 \) is a group with triality \( S = \langle \sigma, \rho \rangle \).
\end{proof}

\subsection{Groups $G_n$ with triality} \label{subsec:general.case}

In this subsection, we prove that the group $G_n$, defined by the construction above, is indeed a group with triality. We begin by introducing some notation. 

Let $I_n = \{1,\dots,n\}$ and $T = \{(i,j,k) \in I_n \mid 1 \leq i < j < k \leq n\}$. 
Clearly, $|T| = {n \choose 3} = \frac{n(n-1)(n-2)}{6}$. 
For each $\alpha = (i,j,k) \in T$, define
\[
A_{\alpha} = \langle a_i, a_j, a_k, b_i, b_j, b_k \rangle.
\]

We also set 

\[
A = \bigoplus_{\alpha \in T} A_{\alpha}.
\]

\begin{theorem}\label{theorem:gn.triality}
	Let $G_n$ be a nilpotent group of class~$3$ with $2n$ generators 
	\[
	\{a_1, \dots, a_n, b_1, \dots, b_n\}
	\]
	satisfying the relations \eqref{rel.1}--\eqref{rel.8}.  
	Let $T$, $A_{\alpha}$, and $A$ be as defined above.  
	If each $A_{\alpha}$ is a nilpotent group of class~$3$ with triality (satisfying the same relations as $G_3$), then there exists an embedding
	\[
	G_n \hookrightarrow A.
	\]
	As a consequence, $G_n$ is a group with triality.
\end{theorem}

\begin{proof}
	Consider $G_n$, $A$ and $A_{\alpha}$ as defined in the theorem statement. 
	Since each $A_{\alpha}$ is a group with triality $S_{\alpha}=\langle \sigma_{\alpha}, \rho_{\alpha} \rangle$, 
	and $A$ is defined as the direct sum of such groups, it follows that $A$ is a group with triality.
	In fact, the direct sum $A=\displaystyle \bigoplus_{\alpha\in T}A_{\alpha}$ carries the triality action
	\[
	S=\langle \sigma,\rho\rangle \quad\text{with}\quad 
	\sigma=(\sigma_{\alpha})_{\alpha\in T},\;\; \rho=(\rho_{\alpha})_{\alpha\in T},
	\]
	acting componentwise.
	
	
	In order to construct the embedding, define
	\[
	\pi \colon G_n \longrightarrow A,\qquad
	\pi(g)=\big(\pi_\alpha(g)\big)_{\alpha\in T},
	\]
	where each $\pi_\alpha\colon G_n\to A_\alpha$ is defined on generators
	$x_i\in\{a_i,b_i\}$ by
	\[
	\pi_{\alpha}(x_i)=
	\begin{cases}
		x_i,& i\in\alpha,\\
		1,& i\notin\alpha,
	\end{cases}
	\]
	and extended uniquely to a homomorphism.
	
	In particular,
	\[
	\ker(\pi_\alpha)=\langle a_j,b_j\mid j\notin\alpha\rangle.
	\]
	
	The fact that $\ker \pi= 1$ follows from the equality 
	\[
	\ker \pi = \bigcap_{\alpha \in T}\ker(\pi_\alpha)
	\]
	and from the hypothesis that $G_n$ is nilpotent of class~$3$.
	Indeed, every element of $G_n$ can be written as a product of generators, their powers, 
	commutators of length two, and commutators of length three. 
	In particular, any nontrivial element involves at most three indices. 
	Hence there exists some $\alpha \in T$ containing these indices, so that 
	$\pi_\alpha$ does not send this element to $1$. 
	Therefore no nontrivial element lies in all kernels, and we conclude that 
	\[
	\bigcap_{\alpha \in T} \ker(\pi_\alpha) = 1.
	\]
	
	Thus $\pi$ is injective, and hence $G_n\hookrightarrow A$. 
	By the First Isomorphism Theorem, we conclude that
	\[
	G_n \cong \pi(G_n).
	\]

	Moreover, the componentwise action $S$ on $A$ preserves $\pi(G_n)$: indeed,
	for each $i$ we have $\sigma(\pi(a_i))=\pi(b_i)$ and $\sigma(\pi(b_i))=\pi(a_i)$
	(and similarly for $\rho$). Hence $\pi(G_n)$ is a group with triality.
	
	Since $G_n \cong \pi(G_n)$, we conclude that $G_n$ is also a group with triality.
\end{proof}

The following remarks will be useful in the subsequent sections.

\begin{remark}\label{rem:notation.embedding}
	For later convenience, we shall also write the embedding obtained in Theorem~\ref{theorem:gn.triality} as
	\[
	G_n \hookrightarrow \prod_{\alpha} G_3^{(\alpha)},
	\]
	since each component $A_{\alpha}$ is isomorphic to a copy of $G_3$. 
	This notation is equivalent to writing $A = \displaystyle\bigoplus_{\alpha} A_{\alpha}$.
\end{remark}

\begin{remark}\label{remark:additional.relations}
	Note that if we impose the additional relations $a_i^4 = b_i^4 = 1$ for all $i = 1,\dots,n$ 
	in the construction of $G_n$ and replace 
	\[
	\mathbb{P}_n = \mathbb{Z}^{2n} \oplus \mathbb{F}_2^m 
	\quad \text{by} \quad 
	\mathbb{P}_n^{(4)} = \mathbb{Z}_4^{\,2n} \oplus \mathbb{F}_2^m,
	\]
	we obtain the same results as in the case where the generators have infinite order, 
	such as the isomorphism between $G_n$ with $\mathbb{P}_n^{(4)}$ and the proof that $G_n$ is a group with triality. 
	
	In this case, the generators $a_i,b_i$ have order~$4$, 
	while all central elements (squares, commutators, and associators), 
	as well as the elements $p_{ij} = [a_i,b_j]$, have order~$2$. 
	Hence the group $G_n$ is finite of order
	\[
	|G_n| = 4^{2n} \times 2^{m} = 2^{4n + m},
	\]
	where $m = 3{n \choose 2} + 2{n \choose 3}$. 
	In particular, for $n = 3$, we obtain $|G_3| = 2^{23}$.
	
	For this group $G_n$ with the additional relations, we will determine in Section \ref{section:loop.corresponding} 
	the corresponding Moufang loop.
\end{remark}


\section{The Variety Generated by Code Loops}\label{section:variety.code.loops}

In this section, we introduce the variety ${\mathcal E}$, defined by identities satisfied by all code loops. The results presented guarantee the existence of a free loop in this variety and summarize fundamental properties of this loop, as well as its relation to code loops. Further details can be found in \cite{GP}.

\begin{definition}
	Let  ${\mathcal E}$ be the variety of Moufang loops with the following identities:
	\begin{align}
		&x^4 = 1, \,\left[x,y \right]^2 =1 , \, (x,y,z)^2 = 1, \nonumber \\
		&\left[x^2,y \right] = 1,\,    \left[\left[x,y\right],t \right] = 1,\, \left[(x,y,z),t \right] = 1, \label{identities.of.the.variety} \\
		&(x^2,y,z) = 1 ,  \,  (\left[x,y \right],z,t) = 1, \;\; ((x,y,z),t,s)=1. \nonumber
	\end{align}
\end{definition}

The following properties concerning associators, commutators, and squares hold in every Moufang loop.

\begin{proposition} \label{prop2.5}Let $F$ be a Moufang loop.
	\begin{enumerate}
		\item If $(x,y,z)^{2}=1$ and all the commutators and associators of $F$ are central, then
		\begin{eqnarray*}
			[xy,z]=[x,z][y,z](x,y,z).
		\end{eqnarray*}
		\item If the commutators and associators of $F$ are  central, then
		\begin{eqnarray*}
			(wx,y,z)=(w,y,z)(x,y,z).
		\end{eqnarray*}
		
		\item  If the squares and commutators of a Moufang loop $F$ are central, then
		\begin{eqnarray*}(xy)^{2}=x^{2}y^{2}[x,y].\end{eqnarray*}
	\end{enumerate}
\end{proposition}

In what follows, we denote by $F_n$ a free loop in $\mathcal{E}$; its existence will be recalled after the next lemma.

\begin{lemma}\label{lemma.center}
	Let $F_{n}$ be a free loop in ${\mathcal E}$, with free generator set $\{x_{1},\dots,x_{n}\}$. Then for all $z \in {\mathcal Z}(F_{n})$, there are $\xi_{1},\dots,\xi_{n}$, $\xi_{ij}$, $\xi_{ijk}\in\{0,1\}$, with $i,j,k=1,\dots,n$ such that
	\begin{equation}\label{eq.center}
		z = \prod_{i=1}^{n}(x_{i}^{2})^{\xi_{i}}.\prod_{i<j}[x_{i},x_{j}]^{\xi_{ij}}.\prod_{i<j<k}(x_{i},x_{j},x_{k})^{\xi_{ijk}}. 
	\end{equation}
\end{lemma}

The structure of the center will be useful when determining the Moufang loop corresponding to $G_n$.

According to \cite[Theorem~2.7, p.~170]{GP}, a free loop of rank $n$ exists in the variety $\mathcal{E}$. 
We denote this loop by $\mathcal{F}_n$.  
Moreover, \cite[Corollary~2.9, p.~171]{GP} proves that $\mathcal{E}$ is generated by the set of all code loops. 
In what follows, we write $\mathcal{F}_n = \langle x_1,\dots,x_n\rangle$. The next proposition presents some properties related to this free loop.

\begin{proposition} [\cite{GP}, Corollary 2.8, p. 171]\par\noindent
	
	\begin{enumerate}
		\item For any code loop $L$ of rank $n$, there is a homomorphism $\varphi: \mathcal{F}_n \to L$ such that $\varphi(\mathcal{F}_n) = L$ and $\text{codim}_{\mathcal{Z}(\mathcal{F}_n)} \ker(\varphi) = 1.$
		\item For all $\mathbb{F}_2$-subspaces $T \subset \mathcal{Z}(\mathcal{F}_n)$ of codimension $1$, there exists a code loop $L(T) = \mathcal{F}_n / T.$
		\item The loop $L(T)$ is a group if and only if $T \supseteq U_{n} = (\mathcal{F}_{n}, \mathcal{F}_{n}, \mathcal{F}_{n}) = \mathcal{F}_{n} \cap U$.
	\end{enumerate}
	
\end{proposition}


\section{The Moufang Loop corresponding to $G_n$}\label{section:loop.corresponding}

In Section \ref{section:group.with.triality}, we constructed a nilpotent group $G_n$ of class 3 with triality and $2n$ generators. As observed in Remark~\ref{remark:additional.relations}, to determine the corresponding Moufang loop we consider $G_n$ with the additional relations $a_i^4 = b_i^4 = 1$ for all $i=1,\dots,n$. Our main goal in this section is to show that the loop associated with this group is precisely the free loop $\mathcal{F}_n$ in the variety $\mathcal{E}$. First, we prove that $\mathcal{F}_n$ can be embedded into a direct product of copies of a free code loop of rank $3$.

\subsection{Embedding of \( \mathcal{F}_n \)}

Consider $M_3$ as a free code loop with minimal generators $a,b$ and $c$. Then, by the properties of code loops, we have that $$M_3 = 
\langle a, b, c, [a,b], [a,c], [b,c], a^2, b^2, c^2, (a,b,c) \rangle$$ and, consequently,  $|M_3| = 2^{10}$.
Let  $T  = \{(i,j,k) \ | \ 1\leq i < j < k \leq n\}$ be as defined before. For $\alpha = (i,j,k) \in T$, we denote  $M_{3}^{\alpha} = \langle x_i, x_j, x_k \rangle$ as a free  code loop in the variety $\mathcal{E}$.

In analogy with the group case (Theorem~\ref{theorem:gn.triality}), 
we now show that the free loop $\mathcal{F}_n$ admits an embedding into a direct product of copies of $M_3$.

\begin{lemma}\label{lemma:loop.embedded}
	Let $\mathcal{E}$ be the variety generated by the set of all code loops,  $\mathcal{F}_n \in \mathcal{E}$ a free loop with $n$ generators. If $M_3$ is a free code loop with $3$ generators, then $\mathcal{F}_n$ can be embedded in a direct product of copies of $M_3$. 
\end{lemma}

\begin{proof} Let  $T  = \{(i,j,k)| 1\leq i < j < k \leq n\}$ and $L_n =  \displaystyle \prod_{\alpha \in T} M_{3}^\alpha$, where $M_{3}^{\alpha} = \langle x_i, x_j, x_k \rangle$ is a free  code loop in the variety $\mathcal{E}$, for each $\alpha = (i,j,k) \in T$. Consider the map $\varphi: \mathcal{F}_n \rightarrow L_n$ defined by $\varphi(x_i) = \displaystyle \prod_{\alpha \in T}\varphi_{\alpha}(x_i)$, where 
	$\varphi_{\alpha}(x_i) = \begin{cases}
		x_i, & i \in \alpha; \\
		1, & i \notin \alpha .
	\end{cases}$
	
	Note that $L_n$ is a code loop of the variety $\mathcal{E}$. While it is not a free loop, it contains a free structure within itself.

	Since $\mathcal{F}_n$ is free, any choice of image for $x_i$ can be extended to a homomorphism.
	
	
	We need to prove that $\ker(\varphi) = 1$. 
	First, observe that $\ker(\varphi) \subset \mathcal{Z}(\mathcal{F}_n)$. 
	Indeed, if $x \in \ker(\varphi)$ and $y \in \mathcal{F}_n$, then $xy = yx$, 
	because $\varphi_{\alpha}(x) = 1$ for all $\alpha = (i,j,k) \in T$, 
	and thus 
	\[
	\varphi(xy) = (1,1,\dots,1)\varphi(y) = \varphi(y)(1,1,\dots,1) = \varphi(yx).
	\]
	
	Let $x \in ker(\varphi)$. According to Lemma \ref{lemma.center}, there exist coefficients $\xi_{i}$, $\xi_{ij}$, $\xi_{ijk}\in\{0,1\}$ such that:

	\begin{equation}
		x = \displaystyle \prod_{1 \leq i \leq n}(x_{i}^{2})^{\xi_{i}}. \displaystyle \prod_{1 \leq i<j \leq n}[x_{i},x_{j}]^{\xi_{ij}}. \displaystyle \prod_{1\leq i<j<k \leq n}(x_{i},x_{j},x_{k})^{\xi_{ijk}}. \label{eq16}
	\end{equation}
	
	When applying $\varphi$ to $x \in ker(\varphi) $, each coordinate $\varphi_{\alpha}(x)$ takes the form:

	
	\begin{equation}
		\varphi_{\alpha}(x) 
		=  (x_{i}^{2})^{\xi_{i}} (x_{j}^{2})^{\xi_{j}} (x_{k}^{2})^{\xi_{k}}
		[x_i, x_j]^{\xi_{ij}} [x_i, x_k]^{\xi_{ik}} [x_j, x_k]^{\xi_{jk}}
		(x_{i},x_{j},x_{k})^{\xi_{ijk}}.
	\end{equation}
	
	If there exists a triple $(i,j,k)$ such that $\xi_{ijk} = 1$, choose $\beta = (i,j,k)$ and consider the projection map $p_{\beta}: L_n \rightarrow M_{3}^{\beta}$. There exists a unique homomorphism $g_{\beta}:\mathcal{F}_n \rightarrow M_{3}^{\beta}$ such that $g_{\beta} = p_{\beta}\circ \varphi$. This gives us the commutative diagram:
	
	\[
	\begin{array}{ccc}
		\mathcal{F}_n  & {\overset{\varphi}{\longrightarrow}} & L_n  \\
		& \underset{\hspace{-0.1in} g_{\beta}}{\searrow} & \downarrow{p_{\beta}} \\
		&  & M_3 ^{\beta} \\
	\end{array}
	\]
	Since $\varphi(x) = 1$, it follows that $g_{\beta}(x) = 1$.

	To show that $x=1$, we will consider two cases:
	
	\textbf{Case 1:} Suppose   $$x = (x_{i}^{2})^{\xi_{i}} (x_{j}^{2})^{\xi_{j}} (x_{k}^{2})^{\xi_{k}}  [x_i, x_j]^{\xi_{ij}} [x_i, x_k]^{\xi_{ik}} [x_j, x_k]^{\xi_{jk}} (x_{i},x_{j},x_{k}).$$

	Then, calculating $g_{\beta}(x)$, we get   $$1 = g_{\beta}(x) = p_{\beta}(\varphi(x)) = (x_i ^2)^{\beta{i}} (x_j ^2)^{\beta{j}} (x_k ^2)^{\beta{k}} [x_i, x_j]^{\beta{ij}}[x_i,x_k]^{\beta{ik}}[x_j,x_k]^{\beta{jk}}(x_{i},x_{j},x_{k}) = x,$$ since  the corresponding exponents are equal.
	
	In this case, the result follows directly.
	
	\textbf{Case 2:} If there are other terms in the expression of $x$ involving indices of more than three generators, suppose there is another triple $\gamma = (i^{'} , j^{'} , k^{'})$ such that $\xi_{i^{'}j^{'}k^{'}} = 1$. Assume, without loss of generality, that: $$x = (x_{i}^{2})^{\xi_{i}} (x_{j}^{2})^{\xi_{j}} (x_{k}^{2})^{\xi_{k}}  [x_i, x_j]^{\xi_{ij}} [x_i, x_k]^{\xi_{ik}} [x_j, x_k]^{\xi_{jk}} (x_{i},x_{j},x_{k})(x_{i^{'}},x_{j^{'}},x_{k^{'}}).$$
	
	In order to conclude that  $x=1$ ,  we  calculate $g_{\beta}(x)$ and $g_{\gamma}((x_{i^{'}},x_{j^{'}},x_{k^{'}}))$. Applying $g_{\beta}$ to $x$ we get: $$(x_i ^2)^{\beta{i}} (x_j ^2)^{\beta{j}} (x_k ^2)^{\beta{k}} [x_i, x_j]^{\beta{ij}}[x_i,x_k]^{\beta{ik}}[x_j,x_k]^{\beta{jk}}(x_{i},x_{j},x_{k}) = 1.$$
	
	Thus, $x$ can be reduced to $x = (x_{i^{'}},x_{j^{'}},x_{k^{'}}).$ At last, we apply $g_{\gamma} = p_{\gamma}\circ \varphi$ to $x = (x_{i^{'}},x_{j^{'}},x_{k^{'}})$ and determine $g_{\gamma}(x) = x = 1$. 
	
	Note that if there are additional terms involving  other generators, it will be necessary to construct homomorphisms $g_{\beta}$ until $x$ is fully reduced to the identity element. 

	If no triple $\beta= (i,j,k)$ exists where $\xi_{ijk} = 1$, we can focus on elements where $\xi_{ij} = 1$ or $\xi_i = 1$ and the analysis will proceed in a similar manner.
\end{proof}

This lemma shows that the free loop $\mathcal{F}_n$ can be described in terms of copies of $M_3$. 
To continue the analogy with the group case, we now determine explicitly the Moufang loop $M_3$ corresponding to $G_3$.

\subsection{Determination of the Moufang Loop \( M_3 \) Corresponding to \( G_3 \)}Let us first review some results provided in \cite{Doro} and \cite{GrishkovZavarnitsine06}.
Let $G$ be a \textit{group with triality} $S = <\sigma, \rho>$. Doro \cite{Doro} proved that  $G$ can be decomposed as follows: 
\( G = H M^{\rho^2} \), where 
\( H = \{ x \in G \mid x^\sigma = x \} \) and \( M = \{ x^{-1} x^\sigma \mid x \in G \} \); 
moreover, \( (M^{\rho^2}, \star) \) is a Moufang loop with multiplication 
\begin{equation}\label{eq:loopdoro}
	x \star y = z  \iff xy = hz,  \,\,\, \mbox{for}\,\, h \in H,\, x,y,z \in M^{\rho^2}.
\end{equation}

The following theorem, due to Grishkov and Zavarnitsine, provides an explicit description of the Moufang loop structure associated with a group with triality. 

\begin{theorem}[\cite{GrishkovZavarnitsine06}, Theorem 1, p.445] \label{the:GZ} Let $G$ be a group with triality and $M = \{x^{-1}x^{\sigma}\,|\,x \in G\}$. Then the set $M$ is a Moufang loop with respect to the multiplication 
	\begin{equation}
		m\cdot n = m^{-\rho}nm^{-\rho^{2}}, \, \, \forall \, \,  m,n \in M.
	\end{equation}
	Moreover, this Moufang loop $(M,\cdot)$ is isomorphic to Doro's loop $(M^{\rho^2}, \star)$ with multiplication given by $(\ref{eq:loopdoro})$.
	
\end{theorem}

\begin{definition}
	For a group with triality $G$, we denote by $L(G)$ the Moufang loop associated to $G$,
	that is,
	\[
	L(G) := \{ x^{-1}x^\sigma \mid x \in G \},
	\]
	with multiplication
	\[
	m\cdot n = m^{-\rho} n m^{-\rho^2}, \qquad m,n\in L(G).
	\]
\end{definition}

\begin{remark}\label{remark.triality} Now, let \( G_3 \) be a nilpotent group of class 3 with 6 generators $ \{a_1, a_2, a_3, b_1,$ $ b_2, b_3\}$, where each generator satisfies \( a_i^4 = b_i^4 = 1 \) and the relations \ref{rel.1} through \ref{rel.8} given in Section \ref{section:group.with.triality}. 
	We suppose $u_{ij} = [a_i, a_j],$ $v_{ij} = [b_i, b_j]$, $p_{ij} = [a_i, b_j]$ for $i < j$ and define $z_{123}$, $t_{123}$ accordingly by these relations. Additionally, the automorphisms $\sigma$ and $\rho$ are defined by Equations \ref{eq.sigma} and \ref{eq.rho}. We have previously proven that $G_3$ is a group with triality with respect to $S= <\sigma, \rho>$. 
	
	Furthermore, the squares of the generators, as well as the commutators and associators, are central of order \(2\), 
	which contributes to the order \( |G_3| = 2^{23} \).
\end{remark}


\begin{theorem}
	Let \( G_3 \) be a group with triality with respect to $S= <\sigma, \rho>$, as defined in Remark $\ref{remark.triality}$.
	Let \( M_3 \) defined as:
	\[
	M_3 = \{x^{-1} x^\sigma \mid x \in G_3\},
	\]
	with respect to the  multiplication law:
	\[
	m \cdot n = m^{-\rho} n m^{-\rho^2}, \quad \text{for all } m, n \in M_3.
	\]
	The set \( M_3 \) is the corresponding Moufang loop to $G_3$ and has the following properties:
	\begin{enumerate}
		\item The loop \( M_3 \) is generated by:
		\[
		\{x_1, x_2, x_3, [x_1, x_2], [x_1, x_3], [x_2, x_3], (x_1, x_2, x_3), x_1^2, x_2^2, x_3^2\},
		\]
		where \( x_i = a_i^{-1} a_i^\sigma \) for \( i = 1, 2, 3 \).
		\item The order of \( M_3 \) is \( |M_3| = 2^{10} \).
		
		\item \( M_3 \) is a free loop in the variety \( \mathcal{E} \) generated by code loops.
	\end{enumerate}
	
\end{theorem}

\begin{proof}
	Let \( M_3 = \{ x^{-1} x^\sigma \mid x \in G_3\} \), where \( G_3 \) is a group with triality with generators \( \{a_1, a_2, a_3, b_1, b_2, b_3\} \) satisfying \( a_i^4 = b_i^4 = 1 \), and automorphisms \( \sigma \) and \( \rho \) such that \( \sigma^2 = \rho^3 = (\sigma \rho)^2 = 1 \).
	
	Since the multiplication in \( M_3 \) is defined by:
	\[
	m \cdot n = m^{-\rho} n m^{-\rho^2}, \quad \text{for all } m, n \in M_3,
	\]
	
	\noindent and the  elements of \( M_3 \) are given by \( x^\sigma x^{-1} \), for all $x \in G_3$  they can be  explicitly described as:
	\[
	x_i = a_i^{-1} a_i^\sigma \quad \text{for } i = 1, 2, 3.
	\]

	Therefore, we can conclude that the loop \( M_3 \) is generated by the elements:
	\[
	\{x_1, x_2, x_3, [x_1, x_2], [x_1, x_3], [x_2, x_3], (x_1, x_2, x_3), x_1^2, x_2^2, x_3^2\}.
	\]

	In order to determine the order of \( M_3 \), consider the subgroup:
	\[
	H = \{h \in G_3 \mid h^\sigma = h\}.
	\]
	
	With direct calculation, we show that this subgroup has order \( |H| = 2^{13} \), as it is generated by: 
	\[
	\{a_i b_i, (a_i b_i)^2, p_{ij}, u_{ij} v_{ij}, t z\}, \quad \text{where } t = t_{123} \text{ and } z = z_{123}.
	\]

	
	Since \( |G_3| = 2^{23} \), it follows that:
	\[
	|M_3| = \frac{|G_3|}{|H|} = \frac{2^{23}}{2^{13}} = 2^{10}.
	\]
	Indeed, by Doro's result, we can decompose $G_3 = H.M^{\rho^2}$, where $M=M_3$, and $(M^{\rho^2}, \star)$ is a Moufang loop  with the  multiplication law given by \ref{eq:loopdoro}.  Moreover, as established by Grishkov and Zavarnitsine in Theorem \ref{the:GZ}, the loops $(M^{\rho^2}, \star)$ and $(M_3, \cdot)$ are isomorphic.

	The structure of \( M_3 \) satisfies the defining identities of the variety \( \mathcal{E} \), namely: \( x^4 = 1 \), \( [x, y]^2 = 1 \), and \( (x, y, z)^2 = 1 \), as well as the centrality conditions. Therefore, \( M_3 \) is a Moufang loop of order \( 2^{10} \), freely generated in the variety \( \mathcal{E} \).
\end{proof}

We conclude that the loop $M_3$ obtained from $G_3$ has the expected structure of a free code loop of rank three. 
This result allows us to obtain the general correspondence between $G_n$ and $\mathcal{F}_n$. 
Indeed, by Theorem \ref{theorem:gn.triality}, $G_n$ can be embedded in a direct product of copies of $G_3$. 
By Lemma~\ref{lemma:loop.embedded}, the free loop $\mathcal{F}_n$ in the variety $\mathcal{E}$ is also embedded in a direct product of copies of $M_3$, 
which are the Moufang loops associated to the corresponding copies of $G_3$. 

\begin{remark}\label{rem:embedding.induced}
	The embedding $G_n \hookrightarrow \displaystyle\prod_\alpha G_3^{(\alpha)}$ naturally induces 
	an embedding $L(G_n)\hookrightarrow \displaystyle\prod_\alpha M_3^{(\alpha)}$, 
	where the automorphisms $\sigma$ and $\rho$ act componentwise on the direct product.
\end{remark}

As a consequence, we obtain the following result:

\begin{corollary}\label{cor:main.result}
	Let \( G_n \) be the nilpotent group of class \( 3 \) with triality and with the additional relations given in Remark $\ref{remark:additional.relations}$, and let \( \mathcal{F}_n \) be the free loop with \( n \) generators in the variety \( \mathcal{E} \) generated by code loops. Then \( \mathcal{F}_n \) is the Moufang loop corresponding to \( G_n \).
\end{corollary}

\begin{proof}
	By the universal property of $\mathcal{F}_n$, there exists a unique homomorphism 
	\[
	\pi:\mathcal{F}_n \to L(G_n)
	\]
	sending each free generator to its image in $L(G_n)$. 
	
	By Lemma~\ref{lemma:loop.embedded}, 
	$\mathcal{F}_n$ embeds in a direct product of copies of $M_3$. So, let $\varphi:\mathcal{F}_n \hookrightarrow \prod_{\alpha} M_3^{(\alpha)}$ be this embedding. Moreover, let
	$\eta:L(G_n)\hookrightarrow \prod_{\alpha} M_3^{(\alpha)}$ be the embedding described as in Remark \ref{rem:embedding.induced}, 
	chosen so that $\eta\circ\pi=\varphi$.  
	Since both $\varphi$ and $\eta$ are injective, $\pi$ is injective. 
	Moreover, since the images of the generators of $\mathcal{F}_n$ generate $L(G_n)$, $\pi$ is surjective.  
	Therefore, $\pi$ is an isomorphism, and $\mathcal{F}_n \cong L(G_n)$.
	
	\[
	\begin{minipage}{0.45\textwidth}
		\centering
		\begin{tikzcd}[column sep=large, row sep=large]
			G_n \arrow[r, hook, "\iota"] \arrow[d] 
			& \displaystyle\prod_{\alpha} G_3^{(\alpha)} \arrow[d] \\
			L(G_n) \arrow[r, hook, "\eta"] 
			& \displaystyle\prod_{\alpha} M_3^{(\alpha)}
		\end{tikzcd}
	\end{minipage}
	\hfill
	\begin{minipage}{0.45\textwidth}
		\centering
		\begin{tikzcd}[column sep=large, row sep=large]
			\mathcal{F}_n \arrow[dr, hook, "\varphi"'] \arrow[r, "\pi"] 
			& L(G_n) \arrow[d, hook, "\eta"] \\
			& \displaystyle\prod_{\alpha} M_3^{(\alpha)}
		\end{tikzcd}
		\\[0.5em]
		\text{with } $\eta \circ \pi = \varphi$
	\end{minipage}
	\]

\end{proof}

	This result completes the construction we aimed to establish: a nilpotent group \( G_n \) of class \( 3 \) with triality, whose corresponding Moufang loop is precisely the free loop \( \mathcal{F}_n \) in the variety \( \mathcal{E} \) generated by code loops.
	
	In the next section, we compare this construction with the one given by Nagy \cite{Nagy}, showing how our result generalizes his.


	
	\section{Comparison with Nagy’s construction}\label{section:comparison}
	
	In \cite{Nagy}, the author constructs a specific group with triality, denoted here by $G'_n$, starting from a given code, 
	whose associated loop is the code loop determined by that code. 
	In this sense, his approach is concrete: each choice of parameters produces one code loop. 
	We omit here the details of his construction. 
	
	Our approach is different in that we build a universal group with triality $G_n$, 
	whose associated loop $\mathcal{F}_n = L(G_n)$ is the free object in the variety $\mathcal{E}$ generated by code loops. 
	From this free loop, every code loop of rank $n$ can be obtained as a quotient 
	$\mathcal{F}_n \twoheadrightarrow L_n$ determined by a characteristic vector $\lambda$. 
	
	The relation between the two constructions can be summarized in the following diagram:
	
	\[
	\begin{tikzcd}[row sep=large, column sep=huge]
		G_n 
		\arrow[r, "\phi"] 
		\arrow[d, "\psi"'] 
		& G'_n 
		\arrow[r, "\theta"] 
		\arrow[d, "g_{\lambda}"'] 
		& L'_n 
		\arrow[dl, "f_{\lambda}", "\cong"'] 
		\\
		\mathcal{F}_n 
		\arrow[r, two heads, "\pi_{\lambda}"'] 
		& L_n &
	\end{tikzcd}
	\]
	
	Here $\psi: G_n \to \mathcal{F}_n$ is the natural map obtained from Corollary~\ref{cor:main.result}, 
	and $\pi_{\lambda}: \mathcal{F}_n \twoheadrightarrow L_n$ is the canonical quotient determined by the characteristic vector $\lambda$.

	On the other side, according to the construction given by Nagy~\cite{Nagy}, 
	the homomorphism $\theta: G'_n \to L'_n$ produces a code loop $L'_n$, 
	and the isomorphism $f_{\lambda}: L'_n \to L_n$ identifies this loop with the one obtained from $\mathcal{F}_n$.

	Thus $g_{\lambda} := f_{\lambda}\circ\theta$ is a natural map $G'_n \to L_n$. 
	Finally, by the universality of $G_n$, there exists a homomorphism $\phi: G_n \to G'_n$, 
	chosen so that the diagram commutes:
	\[
	g_{\lambda}\circ \phi = \pi_{\lambda}\circ \psi.
	\]

	This shows that both constructions lead to the same code loop $L_n$.  
	From our side, $G_n$ gives the free loop $\mathcal{F}_n$, and every code loop $L_n$ is a quotient of $\mathcal{F}_n$.  
	From Nagy’s side, $G'_n$ produces directly a code loop $L'_n$, which is isomorphic to $L_n$.  
	Moreover, this establishes that our construction actually generalizes Nagy’s: 
	while $G'_n$ corresponds to one particular code loop, 
	$G_n$ is the free group with triality from which all such examples arise.  
	In other words, $G'_n$ appears as an image of $G_n$, 
	and Nagy’s code loops appear as quotients of $\mathcal{F}_n$.

	\subsection*{Acknowledgments}
	The authors A. Grishkov, R. M. Pires, and M. Rasskazova thank the National Council for Scientific and Technological Development CNPq (grant 406932/\\2023-9). A. Grishkov was supported by FAPESP (grant 2024/14914-9), CNPq (grant 307593/2023-1), and in accordance with the state task of the IM SB RAS, project FWNF-2022-003. 
	
	
	

	
	
	
	
	
	
	


	\newpage
	\section{Appendix}\label{appendice:tables}
	
	\begin{table}[h]
		\begin{tabular}{|l|l|l|l|}
			\hline
			& $\overline{\alpha}$           & $\overline{\beta}$           & $\overline{\alpha}\cdot \overline{\beta}$                                                                                                                                    \\ \hline
			$1$                                  & $\alpha_1$               & $\beta_1$               & $\alpha_1 + \beta_1$                                                                                                        \\ \hline
			$\vdots$                             & $\vdots$                 & $\vdots$                & $\vdots$                                                                                                                    \\ \hline
			$n$                                  & $\alpha_n$               & $\beta_n$               & $\alpha_n + \beta_n$                                                                                                        \\ \hline
			$n+1$                                & $\alpha_{n+1}$           & $\beta_{n+1}$           & $\alpha_{n+1}+ \beta_{n+1}$                                                                                                                                                                                                                       \\ \hline
			$\vdots$                             & $\vdots$                 & $\vdots$                & $\vdots$                                                                                                                    \\ \hline
			$2n$                                 & $\alpha_{2n}$            & $\beta_{2n}$            & $\alpha_{2n}+\beta_{2n}$                                         \\ \hline
			$2n+1$                               & $\alpha_{12}^1$          & $\beta_{12}^1$          & $\alpha_{12}^1+\beta_{12}^1 + \alpha_{2}\beta_{1}$                                                                          \\ \hline
			$\vdots$                             & $\vdots$                 & $\vdots$                & $\vdots$                                                                                                                                                                               \\ \hline
			$2n + {n \choose 2}$                 & $\alpha_{n-1\,n}^1$      & $\beta_{n-1\,n}^1$      & $\alpha_{n-1\,n}^1 +\beta_{n-1\,n}^1 + \alpha_{n}\beta_{n-1}$                                                                                                                          \\ \hline
			$2n + {n \choose 2} +1$              & $\alpha_{12}^2$          & $\beta_{12}^2$          & $\alpha_{12}^2 + \beta_{12}^2 + \alpha_{2+n}\beta_{1+n}$                                                                    \\ \hline
			$\vdots$                             & $\vdots$                 & $\vdots$                & $\vdots$                                                                                                                    \\ \hline
			$2n + 2{n \choose 2}$                 & $\alpha_{n-1\,n}^2$      & $\beta_{n-1\,n}^2$      & $\alpha_{n-1\,n}^2 + \beta_{n-1\,n}^2 + \alpha_{2n}\beta_{2n-1}$                                                                                                                       \\ \hline
			$2n + 2{n \choose 2}+1$              & $\alpha_{12}^3$          & $\beta_{12}^3$          & $\alpha_{12}^3+\beta_{12}^3+\alpha_{1+n}\beta_{2}+\alpha_{2+n}\beta_{1}$                                                                                                               \\ \hline
			$\vdots$                             & $\vdots$                 & $\vdots$                & $\vdots$                                                                                                                    \\ \hline
			$2n + 3{n \choose 2}$                & $\alpha_{n-1\,n}^3$      & $\beta_{n-1\,n}^3$      & $\alpha_{n-1\,n}^3+\beta_{n-1\,n}^3+\alpha_{2n-1}\beta_{n}+\alpha_{2n}\beta_{n-1}$                                          \\ \hline
			$2n + 3{n \choose 2}+1$              & $\alpha_{123}^1$         & $\beta_{123}^1$         & \begin{tabular}[c]{@{}l@{}}$\alpha_{123}^1+\beta_{123}^1 +  \alpha_{12}^{3}\beta_{3}+\alpha_{13}^{3}\beta_{2}+\alpha_{23}^{3}\beta_{1}$\\ $+ \alpha_{1+n}\beta_{2}\beta_{3}+\alpha_{2+n}\beta_{1}\beta_{3}+\alpha_{3+n}\beta_{1}\beta_{2}$\end{tabular}                                                       \\ \hline
			$\vdots$                             & $\vdots$                 & $\vdots$                & $\vdots$                                                                                                                    \\ \hline
			$2n + 3{n \choose 2}+{n\choose 3}$   & $\alpha_{n-2\,n-1\,n}^1$ & $\beta_{n-2\,n-1\,n}^1$ & \begin{tabular}[c]{@{}l@{}}$\alpha_{n-2\,n-1\,n}^1+\beta_{n-2\,n-1\,n}^1 +  \alpha_{n-2\,n-1}^{3}\beta_{n}$\\ $+\alpha_{n-2\,n}^{3}\beta_{n-1}+\alpha_{n-1\,n}^{3}\beta_{n-2}$\\ $+\alpha_{2n-2}\beta_{n-1}\beta_{n}+\alpha_{2n-1}\beta_{n-2}\beta_{n}$\\ $+\alpha_{2n}\beta_{n-2}\beta_{n-1}$\end{tabular}                                                                                                                                                                                                                                              \\ \hline
			$2n + 3{n \choose 2}+{n\choose 3}+1$ & $\alpha_{123}^2$         & $\beta_{123}^2$         & \begin{tabular}[c]{@{}l@{}}$\alpha_{123}^2+\beta_{123}^2 +  \alpha_{12}^{3}\beta_{3+n}+\alpha_{13}^{3}\beta_{2+n}$\\ $+\alpha_{23}^{3}\beta_{1+n}+\alpha_{1+n}\alpha_{2+n}\beta_{3}+\alpha_{1+n}\alpha_{3+n}\beta_{2}$\\ $+\alpha_{2+n}\alpha_{3+n}\beta_{1} + \alpha_{1+n}(\beta_{2}\beta_{3+n} + \beta_{2+n}\beta_{3})$\\ $+\alpha_{2+n}(\beta_{1}\beta_{3+n} + \beta_{1+n}\beta_{3}) $\\ $ + \alpha_{3+n}(\beta_{1}\beta_{2+n} + \beta_{1+n}\beta_{2})$\end{tabular}                                                                                   \\ \hline
			$\vdots$                             & $\vdots$                 & $\vdots$                & $\vdots$                                                                                                                                                                               \\ \hline
			$2n + 3{n \choose 2}+2{n\choose 3}$  & $\alpha_{n-2\,n-1\,n}^2$ & $\beta_{n-2\,n-1\,n}^2$ & \begin{tabular}[c]{@{}l@{}}$\alpha_{n-2\,n-1\,n}^2+\beta_{n-2\,n-1\,n}^2 +  \alpha_{n-2\,n-1}^{3}\beta_{2n}$\\ $+\alpha_{n-2\,n}^{3}\beta_{2n-1}+\alpha_{n-1\,n}^{3}\beta_{2n-2}$\\ $+\alpha_{2n-2}\alpha_{2n-1}\beta_{n
				}
				$\\ $+\alpha_{2n-2}\alpha_{2n}\beta_{n-1}+\alpha_{2n-1}\alpha_{2n}\beta_{n-2} $\\ $+ \alpha_{2n-2}(\beta_{n-1}\beta_{2n} + \beta_{2n-1}\beta_{n})$\\ $+\alpha_{2n-1}(\beta_{n-2}\beta_{2n} + \beta_{2n-2}\beta_{n}) $\\ $ +\alpha_{2n}(\beta_{n-2}\beta_{2n-1} + \beta_{2n-2}\beta_{n-1})$\end{tabular} \\ \hline
		\end{tabular}
		\caption{Multiplication in $\mathbb{P}_n$}
		\label{tab:mult.pn}
	\end{table}


	\begin{table}[]
		\begin{tabular}{|l|l|l}
			\cline{1-2}
			& $(\overline{\alpha}\cdot \overline{\beta})\cdot \overline{\gamma}$                    &  \\ 
			
			\cline{1-2}
			$1$         &       $\alpha_1+ \beta_1 + \gamma_1$    &  \\ 
			
			\cline{1-2}
			
			$\vdots$                             & $\vdots$ \\ 
			
			\cline{1-2}
			$n$         &  $\alpha_n+ \beta_n + \gamma_n$         &  \\ 
			
			\cline{1-2}
			$n + 1$         &   $\alpha_{n+1}+ \beta_{n+1} + \gamma_{n+1}$         &  \\ 
			
			\cline{1-2}
			
			$\vdots$                             & $\vdots$ \\ 
			
			\cline{1-2}
			$2n$         &   $\alpha_{2n}+\beta_{2n} + \gamma_{2n}$       &  \\ 
			
			\cline{1-2}
			$2n+1$         &   $\alpha_{12}^1+\beta_{12}^1 + \alpha_{2}\beta_{1} + \gamma_{12}^1  + \alpha_{2}\gamma_{1} + \beta_{2}\gamma_{1}$        &  \\ 
			
			\cline{1-2}
			
			$\vdots$                             & $\vdots$ \\ 
			
			\cline{1-2}   
			
			$2n + {n \choose 2}$                 & \begin{tabular}[c]{@{}l@{}}$ \alpha_{n-1\, n}^1 + \beta_{n-1\,n}^1 + \alpha_{n} \beta_{n-1}$\\ $ + \gamma_{n-1\,n}^1 + \alpha_{n} \gamma_{n-1} + \beta_n \gamma_{n-1}$\end{tabular}                                                                                                                                                                                                                                                                                                      &  \\ \cline{1-2}
			$2n + {n \choose 2} +1$              & \begin{tabular}[c]{@{}l@{}}$\alpha_{12}^2 + \beta_{12}^2 + \alpha_{2+n}\beta_{1+n} + \gamma_{12}^2$\\ $ + \alpha_{n+2} \gamma_{n+1} + \beta_{n+2} \gamma_{n+1}$\end{tabular}                                                                                                                                                                                                                                                                                                                                                                &  \\ \cline{1-2}
			$\vdots$                             & $\vdots$ \\ \cline{1-2}
			$2n + 2{n \choose 2}$                 & \begin{tabular}[c]{@{}l@{}}$\alpha_{n-1\,n}^2 + \beta_{n-1\,n}^2 + \alpha_{2n} \beta_{2n-1}$\\ $ + \gamma_{n-1\,n}^2 + \alpha_{2n} \gamma_{2n-1} + \beta_{2n} \gamma_{2n-1} $\end{tabular}                                                                                                                                                                                                                                                                                                                                                                                                                                    &  \\ \cline{1-2}
			
			$2n + 2{n \choose 2}+1$              & \begin{tabular}[c]{@{}l@{}}$\alpha_{12}^3 + \beta_{12}^3 + \alpha_{1+n} \beta_2 + \alpha_{2+n} \beta_1$\\ $ + \gamma_{12}^3 + \alpha_{n+1} \gamma_2 + \beta_{n+1} \gamma_2 + \alpha_{n+2} \gamma_1 + \beta_{n+2} \gamma_1$
			\end{tabular}                                                                                                                                                                                                                                                         &  \\ \cline{1-2}
			$\vdots$                             & $\vdots$ \\ \cline{1-2}
			$2n + 3{n \choose 2}$                &                         
			\begin{tabular}[c]{@{}l@{}} $\alpha_{n-1\,n}^3+\beta_{n-1\,n}^3 +\alpha_{2n-1}\beta_{n}$    \\ $+\alpha_{2n}\beta_{n-1}  +\gamma_{n-1\,n}^3 +\alpha_{2n-1}\gamma_{n}  $\\ 
				$+ \beta_{2n-1}\gamma_{n}   +\alpha_{2n}\gamma_{n-1} +\beta_{2n}\gamma_{n-1}$\\ \end{tabular}
			
			&  \\ \cline{1-2}
			
			$2n + 3{n \choose 2}+1$              & \begin{tabular}[c]{@{}l@{}}$\alpha_{123}^1 + \beta_{123}^1+\alpha_{12}^3 \beta_3 + \alpha_{13}^2 \beta_2 + \alpha_{23}^3 \beta_1$\\ $ + \alpha_{1+n} \beta_2 \beta_3 + \alpha_{2+n} \beta_1 \beta_3 + \alpha_{3+n} \beta_1 \beta_2 + \gamma_{123}^1$\\ $+(\alpha_{12}^3 + \beta_{12}^3 + \alpha_{1+n}\beta_2 + \alpha_{2+n}\beta_1) \gamma_3 + (\alpha_{13}^3$\\ $ + \beta_{13}^3 + \alpha_{1+n} \beta_3 + \alpha_{3+n}\beta_1) \gamma_2$\\ $+(\alpha_{23}^3 + \beta_{23}^3 + \alpha_{2+n} \beta_3 + \alpha_{3+n} \beta_2) \gamma_1$\\ $ + \alpha_{n+1} \gamma_2 \gamma_3 + \beta_{n+1} \gamma_2 \gamma_3 + \alpha_{n+2} \gamma_1 \gamma_3 + \beta_{n+2} \gamma_1 \gamma_3$\\ $+ \alpha_{n+3} \gamma_1 \gamma_2 + \beta_{n+3} \gamma_1 \gamma_2$\end{tabular}                                                                                                                                              &  \\ \cline{1-2}
			$\vdots$                             & $\vdots$ \\ \cline{1-2}
			$2n + 3{n \choose 2}+{n\choose 3}$   & \begin{tabular}[c]{@{}l@{}}$\alpha_{n-2\,n-1\,n}^1 + \beta_{n-2\,n-1\,n}^1 + \alpha_{n-2\,n-1}^3 \beta_n $\\ $ + \alpha_{n-2\,n}^3 \beta_{n-1} + \alpha_{n-1\,n}^3 \beta_{n-2} + \alpha_{n-2} \beta_{n-1} \beta_{n}$\\ $ + \alpha_{2n - 1} \beta_{n-2} \beta_{n} + \alpha_{2n} \beta_{n-2} \beta_{n-1} + \gamma_{n-2\,n-1\,n}^1 $\\ $+(\alpha_{n-2\,n-1}^3 + \beta_{n-2\,n-1}^3 + \alpha_{2n-2}\beta_{n-1} + \alpha_{2n-1} \beta_{n-2})\gamma_n$\\ $ +(\alpha_{n-2\,n}^3 + \beta_{n-2\,n}^3 + \alpha_{2n-2}\beta_n + \alpha_{2n}\beta_{n-2})\gamma_{n-1}$\\ $+ (\alpha_{n-1\,n}^3 + \beta_{n-1\,n}^3 + \alpha_{2n-1}\beta_n + \alpha_{2n}\beta_{n-1})\gamma_{n-2}$\\ $+(\alpha_{2n-2} + \beta_{2n-2}) \gamma_{n-1} \gamma_n + (\alpha_{2n-1} + \beta_{2n-1}) \gamma_{n-2} \gamma_n$\\ $ + (\alpha_{2n} + \beta_{2n}) \gamma_{n-2} \gamma_{n-1} $\end{tabular}                                                                                                                         &  \\ \cline{1-2}
		\end{tabular}\\
		
		\caption{Multiplication in $\mathbb{P}_n$: general formulas for the product of three elements (Part I).}
		\label{tab:assoc.pn.parte1}
		
	\end{table}
	
	\begin{table}[]
		\begin{tabular}{|l|l|l}
			\cline{1-2}
			& $(\overline{\alpha}\cdot \overline{\beta})\cdot \overline{\gamma}$                                                          &  \\ \cline{1-2}
			$2n + 3{n \choose 2}+{n\choose 3}+1$ & \begin{tabular}[c]{@{}l@{}}$\alpha_{123}^2 + \beta_{123}^2 + \alpha_{12}^3 \beta_{3+n} + \alpha_{13}^2 \beta_{2+n} + \alpha_{23}^3 \beta_{1+n}$\\ $ + \alpha_{1+n} \alpha_{2+n}\beta_3 + \alpha_{1+n}\alpha_{3+n}\beta_2 $\\ $+ \alpha_{2+n} \alpha_{3+n} \beta_1 + \alpha_{1+n}(\beta_2 \beta_{3+n} + \beta_{2+n} \beta_{3})$\\ $ + \alpha_{2+n}(\beta_1 \beta_{3+n} + \beta_{1+n} \beta_3) + \alpha_{3+n} (\beta_1 \beta_{2+n} + \beta_{1+n}\beta_2)$\\ $ + \gamma_{123}^2 +(\alpha_{12}^3 + \beta_{12}^3 + \alpha_{1+n}\beta_2 + \alpha_{2+n}\beta_1) \gamma_{3+n}$\\ $ + (\alpha_{13}^3 + \beta_{13}^3 + \alpha_{1+n} \beta_3 + \alpha_{3+n} \beta_1) \gamma_{2+n}$\\ $ +(\alpha_{23}^3 + \beta_{23}^3 + \alpha_{2+n}\beta_3 + \alpha_{3+n}\beta_2) \gamma_{1+n} + (\alpha_{n+1} $\\ $+ \beta_{n+1})(\alpha_{n+2} + \beta_{n+2)}\gamma_3 +(\alpha_{n+1} + \beta_{n+1})(\alpha_{n+3} $\\ $+ \beta_{n+3)}\gamma_2 + (\alpha_{n+2} + \beta_{n+2})(\alpha_{n+3} + \beta_{n+3})\gamma_1 +(\alpha_{n+1}$\\ $ + \beta_{n+1})(\gamma_2 \gamma_{3+n} + \gamma_{2+n}\gamma_3) + (\alpha_{2+n}$\\ $ + \beta_{2+n})(\gamma_1 \gamma_{3+n} + \gamma_{1+n}\gamma_3)$\\ $+ (\alpha_{3+n} + \beta_{3+n})(\gamma_1 \gamma_{2+n} + \gamma_{1+n}\gamma_2)$ \end{tabular}                                                                                              &  \\ \cline{1-2}
			$\vdots$                             & $\vdots$ \\ \cline{1-2}
			$2n + 3{n \choose 2}+2{n\choose 3}$  & \begin{tabular}[c]{@{}l@{}}$ \alpha_{(n-2)(n-1)n}^2 + \beta_{(n-2)(n-1)n}^2 + \alpha_{(n-2)(n-1)}^3 \beta_{2n} + \alpha_{(n-2)n}^3 \beta_{2n-1}$\\ $ + \alpha_{(n-1)n}^3 \beta_{2n-2} + \alpha_{2n-2}\alpha_{2n-1}\beta_{n} + \alpha_{2n-2}\alpha_{2n}\beta_{n-1}$\\ $ + \alpha_{2n-1}\alpha_{2n}\beta_{n-2} + \alpha_{2n-2}(\beta_{n-1}\beta_{2n} + \beta_{2n-1}\beta_{n}) $\\ $+ \alpha_{2n-1}(\beta_{n-2}\beta_{2n} + \beta_{2n-2}\beta_{n}) + \alpha_{2n}(\beta_{n-2}\beta_{2n-1} $\\ $+ \beta_{2n-2}\beta_{n-1}) + \gamma_{(n-2)(n-1)n}^2$ \\ $+ \alpha_{(n-2)(n-1)}^3 \gamma_{2n} + \beta_{(n-2)(n-1)}^3 \gamma_{2n} $\\ $+ \alpha_{(n-2)n}^3 \gamma_{2n-1} + \beta_{(n-2)n}^3 \gamma_{2n-1} + \alpha_{(n-1)n}^3 \gamma_{2n-2}$\\ $ + \beta_{(n-1)n}^3 \gamma_{2n-2} + \alpha_{2n-2} \beta_{n-1}\gamma_{2n} + \alpha_{2n-1} \beta_{n-2} \gamma_{2n}$\\ $ + \alpha_{2n-2}\beta_{n}\gamma_{2n-1} + \alpha_{2n} \beta_{n-2} \gamma_{2n-1} $\\ $+ \alpha_{2n-1}\beta_{n}\gamma_{2n-2} + \alpha_{2n} \beta_{n-1}\gamma_{2n-2}$ \\ $+ (\alpha_{2n-2} + \beta_{2n-2})(\alpha_{2n-1}$\\ $ + \beta_{2n-1})\gamma_{n} + (\alpha_{2n-2} + \beta_{2n-2})(\alpha_{2n} + \beta_{2n})\gamma_{n-1} $\\ $+ (\alpha_{2n-1} + \beta_{2n-1})(\alpha_{2n} + \beta_{2n})\gamma_{n-2} + (\alpha_{2n-2} $\\ $+ \beta_{2n-2})(\gamma_{n-1}\gamma_{2n} + \gamma_{2n-1}\gamma_{n}) + (\alpha_{2n-1}$\\ $ + \beta_{2n-1})(\gamma_{n-2}\gamma_{2n} + \gamma_{2n-2}\gamma_{n}) $\\ $+(\alpha_{2n} + \beta_{2n})(\gamma_{n-2}\gamma_{2n-1} + \gamma_{2n-2}\gamma_{n-1}) $ \end{tabular} &  \\ \cline{1-2}
		\end{tabular}
		\caption{Multiplication in $\mathbb{P}_n$: general formulas for the product of three elements (Part II).}
		\label{tab:assoc.pn.parte2}
	\end{table}

	
	
\end{document}